\documentclass[lettersize,journal]{IEEEtran}
\usepackage{amsmath,amsfonts,amssymb,amsthm}
\usepackage{algorithmic}
\usepackage{algorithm}
\usepackage{array}
\usepackage{accents,calc}
\usepackage[caption=false,font=footnotesize]{subfig}
\usepackage{textcomp}
\usepackage{stfloats}
\usepackage{url}
\usepackage{verbatim}
\usepackage{graphicx}
\usepackage{siunitx}
\usepackage{cite}
\usepackage[textsize=footnotesize]{todonotes}
\usepackage{bbm}
\usepackage{booktabs} 
\usepackage{nomencl}
\usepackage{xcolor} 
\definecolor{LightGray}{gray}{0.9}

\usepackage{enumitem}

\usepackage{tikz}
\usetikzlibrary{shapes.geometric}

\newtheorem{theorem}{Theorem}
\newtheorem{lemma}{Lemma}

\newcommand{\ubar}[1]{\underline{#1}}

\newcommand{\revision}[2]{#2}

\newcommand{\x}{\mathbf{x}}

\newcommand{\Mpb}{M_{\text{pb}}}
\newcommand{\Mres}{M_{\text{r}}}
\newcommand{\Mth}{M_{\text{th}}}

\newcommand{\pd}{\mathbf{d}}
\newcommand{\pfmax}{\mathbf{\bar{f}}}
\newcommand{\pfmin}{\mathbf{\ubar{f}}}
\newcommand{\pg}{\mathbf{p}}
\newcommand{\pgmax}{\mathbf{\bar{p}}}
\newcommand{\res}{\mathbf{r}}
\newcommand{\resmax}{\mathbf{\bar{r}}}
\newcommand{\xith}{\xi_{\text{th}}}

\newcommand{\xir}{\xi_{\text{r}}}

\usepackage{amsmath}

\DeclareMathOperator*{\argmin}{arg-min}

\begin{document}

\title{End-to-End Feasible Optimization Proxies\\ for Large-Scale Economic Dispatch}

\author{Wenbo Chen,~\IEEEmembership{Student Member,~IEEE}, Mathieu Tanneau, Pascal Van Hentenryck,~\IEEEmembership{Member,~IEEE}
}

\markboth{IEEE TRANSACTIONS ON POWER SYSTEMS}%
{Shell \MakeLowercase{\textit{et al.}}: A Sample Article Using IEEEtran.cls for IEEE Journals}


\maketitle

\begin{abstract}
The paper proposes a novel End-to-End Learning and Repair (E2ELR)
architecture for training optimization proxies for economic dispatch
problems. E2ELR combines deep neural networks with closed-form,
differentiable repair layers, thereby integrating learning and
feasibility in an end-to-end fashion. E2ELR is also trained with
self-supervised learning, removing the need for labeled data and the
solving of numerous optimization problems offline. E2ELR is evaluated
on industry-size power grids with tens of thousands of buses using an
economic dispatch that co-optimizes energy and reserves.  The results
demonstrate that the self-supervised E2ELR achieves state-of-the-art
performance, with optimality gaps that outperform other baselines by
at least an order of magnitude.
\end{abstract}


\section*{Nomenclature}

\newlist{abbrv}{itemize}{1}
\setlist[abbrv,1]{label=,labelwidth=0.5in,align=parleft,itemsep=0.1\baselineskip,leftmargin=!}

\subsection{Sets and indices}

\begin{abbrv}
    \item[$i \in \mathcal{N}$] buses
    \item[$e \in \mathcal{E}$] branches
    \item[$g \in \mathcal{G}$] generators
\end{abbrv}

\subsection{Variables}

\renewcommand{\pg}{\mathbf{p}}
\newcommand{\pgbar}{\mathbf{\bar{p}}}

\begin{abbrv}
    \item[$p_{g}$] Energy dispatch of generator $g$
    \item[$r_{g}$] Reserve dispatch of generator $g$
    \item[$\xi_{e}$] Thermal limit violation on branch $e$
\end{abbrv}

\subsection{Parameters}

\newcommand{\PTDF}{\Phi}

\begin{abbrv}
    \item[$d_{i}$] Active power demand at bus $i$
    \item[$c_{g}$] Production cost function of generator $g$
    \item[$\ubar{f}_{e}, \bar{f}_{e}$] Lower and upper thermal limits on branch $e$
    \item[$\Mth$] Thermal violation penalty cost
    \item[$\bar{p}_{g}$] Maximum output of generator $g$
    \item[$\bar{r}_{g}$] Maximum reserve of generator $g$
    \item[$R$] Minimum reserve requirement
    \item[$\PTDF$] PTDF matrix
    \item[$\mathbf{0}$] Vector of all zeros
    \item[$\mathbf{e}$] Vector of all ones
\end{abbrv}

\section{Introduction}
\label{sec:intro}

The optimal power flow (OPF) is a fundamental problem in power systems
operations.  Its linear approximation, the DC-OPF model, underlies
most electricity markets, especially in the US.  For instance, MISO
uses a security-constrained economic dispatch (SCED) in their
real-time markets, which has the DC-OPF model at its core
\cite{BPM_002}.
\revision{}{The continued growth in renewable and distributed energy resources (DERs) has caused an increase in operational uncertainty.
This creates need for new methodologies and operating practices that can quantify and manage risk in real time \cite{MISO2023_ReliabilityImperative,Stover2023_ReliabilityRiskMetrics}.
Nevertheless, quantifying the operational risk of a transmission system requires executing in the order of $10^{3}$ Monte Carlo (MC) simulations, each requiring the solution of multiple (typically in the hundreds) economic dispatch problems \cite{Stover2022_JITRALF,Stover2023_ReliabilityRiskMetrics}.
To fit within the constraints of actual operations, real-time risk assessment would thus require solving hundreds of thousands of economic dispatch problems within a matter of seconds.
This represents a 100,000x speedup over current optimization technology, which can typically solve economic dispatch problems in a few seconds.}

In recent years, there has been a surge of interest, both from the
power systems and Machine-Learning (ML) communities, in developing
optimization proxies for OPF, i.e., ML models that approximate the
input-output mapping of OPF problems.  The main idea is that, once
trained, these proxies can be used to generate high-quality solutions
orders of magnitude faster than traditional optimization solvers.
This capability allows to evaluate a large number of scenarios fast,
thereby enabling real-time risk assessment.

There are however three major obstacles to the deployment of
optimization proxies: \revision{}{feasibility, scalability and adaptability}.
First, ML models are not guaranteed to satisfy
the physical and engineering constraints of the model.  This causes
obvious issues if, for instance, the goal is to use a proxy to
evaluate whether the system is able to operate in a safe state.
Several approaches have been proposed in the past to address the
feasibility of ML predictions, each of which has some advantages
and limitations.  Second, most results in ML for power systems has
considered systems with up to 300 buses, far from the size of actual
systems that contain thousands of buses. Moreover, the resulting
techniques have not been proven to scale or be accurate enough on
large-scale networks.  Third, power grids are not static: the
generator commitments and the grid topology evolve over time,
typically on an hourly basis \cite{Chen2022_Learning4SCED}.
\revision{}{Likewise, the output of renewable generators such as wind and solar generators is non-stationary, which may degrade the performances of ML models.}
Therefore, it is important to be able to (re)train models fast,
typically within a few hours at most.  This obviously creates a
computational bottleneck for approaches that rely on the offline
solving of numerous optimization problems to generate training data.

\revision{}{Despite significant interest from the community, no approach has thus far addressed all three challenges (feasibility, scalability and adaptability) in a unified way.}
This paper addresses \revision{}{this gap} by proposing an {\em
  End-to-End Learning and Repair} (E2ELR) architecture for a
(MISO-inspired) economic dispatch (ED) formulation where the prediction and feasibility
restoration are integrated in a single ML pipeline and trained
jointly. The E2ELR-ED architecture guarantees feasibility, i.e., it
always outputs a feasible solution to the economic dispatch.  The E2ELR-ED
architecture achieves these results through {\em closed-form repair
  layers}. Moreover, through the use of {\em self-supervised
  learning}, the E2ELR-ED architecture is scalable, both for
training and inference, and \revision{}{is found to produce} near-optimal solutions to systems
with tens of thousands of buses.
\revision{}{Finally, the use of self-supervised learning also avoids the costly data-generation process of supervised-learning methods.
This allows to re-train the model fast, smoothly adapting to changing operating parameters and system conditions.}
The contributions of the paper can therefore be summarized as follows:
\begin{itemize}

\item \revision{}{The paper proposes a novel E2ELR architecture with closed-form, differentiable repair layers, whose output is guaranteed to satisfy power balance and reserve requirements.
This is the first architecture to consider reserve requirements, and to offer feasibility guarantees without restrictive assumptions.}

\item \revision{}{The E2ELR model is trained in an end-to-end fashion that combines learning and feasibility restoration, making the approach scalable to large-scale systems.
This contrasts to traditional approaches that restore feasibility at inference time,
which increases computing time and induces significant losses in accuracy.}

\item \revision{}{The paper proposes self-supervised learning to eliminate the need for offline generation of training data.
The proposed E2ELR can thus be trained from scratch in under an hour, even for large systems.
This allows to adapt to changing system conditions by re-training the model periodically.}

\item \revision{}{The paper conducts extensive numerical experiments on systems with up to 30,000 buses, a 100x increase in size compared to most existing studies.
The results demonstrate that the proposed self-supervised E2ELR architecture is highly scalable and outperforms other approaches.}
\end{itemize}

\noindent
The rest of the paper is organized as follows.  Section
\ref{sec:literature} surveys the relevant literature. Section
\ref{sec:overview} presents an overview of the E2ELR architecture and
contrast it with existing approaches.  Section \ref{sec:layers}
presents the problem formulation and the E2ELR architecture in detail.
Section \ref{sec:trainig_methodology} presents supervised and
self-supervised training.  Section \ref{sec:experiments} describes the
experiment setting, and Section \ref{sec:results} reports numerical
results.  Section \ref{sec:conclusion} concludes the paper and
discusses future research directions.

\section{Related works}
\label{sec:literature}

\subsection{Optimization Proxies for OPF}
\label{sec:literature:proxies}

The majority of the existing literature on OPF proxies employs
Supervised Learning (SL) techniques. Each data point $(x, y)$ consists
of an OPF instance data ($x$) and its corresponding solution ($y$).
The training data is obtained by solving a large number --usually tens
of thousands-- of OPF instances offline.  The SL paradigm has
successfully been applied both in the linear DCOPF
\cite{Ng2018_StatisticalLearningDCOPF,pan2020deepopf,Nellikkath2021_PINN-DCOPF,zhao2022ensuring,Chen2022_Learning4SCED,stratigakos2023interpretable,Ferrando2023_PERFORM_NYU}
and nonlinear, non-convex \revision{}{ACOPF
\cite{Guha2019_ML4ACOPF,Zamzam2020_LearningOPF,Owerko2020_OPFusingGNN,fioretto2020predicting,chatzos2021spatial,donti2021dc3,nellikkath2022_PINN-ACOPF,pan2022deepopf,pan2022deepopf-AL,Zhou2022_DeepOPF-FT,liu2022topology,falconer2022leveraging,Owerko2022_unsupervisedOPFusingGNN,Pham2022_ReduceOPFwithGNN,gao2023physics,park2023compact,mitrovic2023data,gupta2022dnn}}
settings.
In almost all the above references, the generator
commitments and grid topology are assumed to be fixed, with
electricity demand being the only source of variability.  Therefore,
these OPF proxies must be re-trained regularly to capture the hourly
changes in commitments and topology that occur in real-life operations
\cite{Chen2022_Learning4SCED}.  
In a SL setting, this comes at a high computational cost because of the need to re-generate training data.
Recent works consider active sampling techniques to reduce this burden
\cite{Klamkin2022_ActiveBucketizedSampling,Hu2023_OPFWorthLearning}.
\revision{}{References \cite{Owerko2020_OPFusingGNN,Owerko2022_unsupervisedOPFusingGNN,Pham2022_ReduceOPFwithGNN,gao2023physics} consider graph neural network (GNN) architectures to accommodate topology changes, however, numerical results are only reported on small networks with at most 300 buses.}

Self-Supervised Learning (SSL) has emerged as an alternative to SL
that does not require labeled data
\cite{Huang2021_DeepOPF-NGT,wang2022fast,park2023self}.  Namely,
training OPF proxies in a self-supervised fashion does \emph{not}
require the solving of any OPF instance offline, thereby removing the
need for (costly) data generation.  In \cite{Huang2021_DeepOPF-NGT},
the authors train proxies for ACOPF where the training loss consists
of the objective value of the predicted solution, plus a penalty term
for constraint violations.  A similar approach is used in
\cite{wang2022fast} in conjunction with Generative Adversarial
Networks (GANs).  More recently, Park et al \cite{park2023self}
jointly train a primal and dual network by mimicking an Augmented
Lagrangian algorithm. Predicting Lagrange multipliers allows for
dynamically adjusting the constraint violation penalty terms in the loss
function.  Current results suggest that SSL-based proxies can match
the accuracy of SL-based proxies.

\subsection{Ensuring Feasibility}
\label{sec:literature:feasibility}

One major limitation of ML-based OPF proxies is that, in general, the
predicted OPF solution violates physical and engineering constraints
that govern power flows and ensure safe operations.  To alleviate this
issue, \cite{Zamzam2020_LearningOPF,zhao2022ensuring} use a
restricted OPF formulation to generate training data, wherein the OPF
feasible region is artificially shrunk to ensure that training data
consists of interior solutions.  In \cite{zhao2022ensuring}, this
strategy is combined with a verification step (see also
\cite{Venzke2021_DNNVerificationPowerSystems}) to ensure the trained
models have sufficient capacity to reach a universal approximation.
Nevertheless, this requires solving bilevel optimization problems,
which is very cumbersome: \cite{zhao2022ensuring} reports training
times in excess of week for a 300-bus system.  In addition, it may not
be possible to shrink the feasible region in general, e.g., when
the lower and upper bounds are the same. 


In the context of DCOPF,
\cite{Ng2018_StatisticalLearningDCOPF,Chen2022_Learning4SCED,stratigakos2023interpretable,Ferrando2023_PERFORM_NYU}
exploit the fact that an optimal solution can be quickly recovered
from an (optimal) active set of constraints.  A combined
classification-then-regression architecture is proposed in
\cite{Chen2022_Learning4SCED}, wherein a classification step
identifies a subset of variables to be fixed to their lower or upper
bound, thus reducing the dimension of the regression task.  In
\cite{Ng2018_StatisticalLearningDCOPF,Ferrando2023_PERFORM_NYU}, the
authors predict a full active set, and recover a solution by solving a
system of linear equations.  Similarly,
\cite{stratigakos2023interpretable} combine decision trees and active
set-based affine policies.  Importantly, active set-based approaches
may yield infeasible solutions when active constraints are incorrectly
classified \cite{Ferrando2023_PERFORM_NYU}.  Furthermore, correctly
identifying an optimal active set becomes harder as problem size
increases.
    
A number of prior work have investigated physics-informed models
(e.g.,
\cite{Zamzam2020_LearningOPF,fioretto2020predicting,chatzos2021spatial,Chen2022_Learning4SCED,pan2022deepopf,nellikkath2022_PINN-ACOPF,Huang2021_DeepOPF-NGT,wang2022fast}).
This approach augments the training loss function with a term that
penalizes constraint violations, and is efficient at reducing -- but
not eliminating -- constraint violations.  To better balance
feasibility and \cite{fioretto2020predicting,chatzos2021spatial}
dynamically adjust the penalty coefficient using ideas from Lagrangian
duality.  In a similar fashion, \cite{park2023self} uses a primal and
a dual networks: the latter predicts optimal Lagrange multipliers,
which inform the loss function used to train the former.
    

Although physics-informed models generally exhibit lower constraint
violations, they still do not produce feasible solutions.  Therefore,
several works combine an (inexact) OPF proxy with a repair step that 
\revision{feasibility restoration step}{restores feasibility}.
A projection step is used in \cite{pan2020deepopf,fioretto2020predicting}, wherein the (infeasible)
predicted solution is projected onto the feasible set of OPF.  In
DCOPF, this projection is a convex (typically linear or quadratic)
problem, whereas the load flow model used in
\cite{fioretto2020predicting} for ACOPF is non-convex.  Instead of a
projection step, \cite{Zamzam2020_LearningOPF} uses an AC power flow
solver to recover voltage angles and reactive power dispatch from
predicted voltage magnitudes and generator dispatches.  This is
typically (much) faster than a load flow.  However, while the
resulting solution satisfies the power flow equations (assuming the
solver converges), it may not satisfy all engineering constraints such
as the thermal limits of the lines.  Finally,
\cite{Taheri2023_RestoringACOPFFeasibility} uses techniques from state
estimation to restore feasibility for ACOPF problems.  This approach
has not been applied in the context of OPF proxies.

The development of implicit differentiable layers
\cite{Agrawal2019_CVXLayer} makes it possible to embed feasibility
restoration inside the proxy architecture itself, thereby removing the
need for post-processing
\cite{donti2021dc3,kim2022projection,Li2022_LOOP-LC}.  This allows to
train models in an \emph{end-to-end} fashion, i.e., the predicted
solution is guaranteed to satisfy constraints.  For instance,
\cite{kim2022projection} implement the aforementioned projection step
as an implicit layer.  However, because they require solving an
optimization problem, these implicit layers incur a very high
computational cost, both during training and testing.  Equality
constraints can also be handled implicitly via so-called constraint
completion \cite{donti2021dc3,Li2022_LOOP-LC}.  Namely, a set of
independent variables is identified, and dependent variables are
recovered by solving the corresponding system of equations, thereby
satisfying equality constraints by design.  Note that constraint
completion requires the set of independent variables to be the same
across all instances, which may not hold in general.  For instance,
changes in generator commitments and/or grid topology may introduce
dependencies between previously-independent variables. The difference
between \cite{donti2021dc3} and \cite{Li2022_LOOP-LC} lies in the
treatment of inequality constraints.  On the one hand,
\cite{donti2021dc3} replace a costly implicit layer with cheaper
gradient unrolling, which unfortunately does not guarantee
feasibility.  On the other hand, \cite{Li2022_LOOP-LC} use gauge
functions to define a one-to-one mapping between the unit hypercube,
which is easy to enforce with sigmoid activations, and the set of
feasible solutions, thereby guaranteeing feasibility.  Nevertheless,
the latter approach is valid only under restrictive assumptions: all
constraints are convex, the feasible set is bounded, and a strictly
feasible point is available for each instance.

\subsection{Scalability Challenges}
\label{sec:literature:scalability}

There exists a significant gap between the scale of actual power
grids, and those used in most academic studies: the former are
typically two orders of magnitude larger than the latter.  On the one
hand, actual power grids comprise thousands to tens of thousands of
buses \cite{Josz2016_ACOPF_PegaseRTE,Holzer2022_MISO_SFT}.  On the
other hand, most academic studies only consider small, artificial
power grids with no more than 300 buses.  Among the aforementioned
works, \cite{Ferrando2023_PERFORM_NYU} considers a synthetic NYISO
grid with 1814 buses, and only
\cite{chatzos2021spatial,Chen2022_Learning4SCED,park2023compact}
report results on systems with more than 6,000 buses.  This
discrepancy makes it difficult to extrapolate most existing findings
to scenarios encountered in the industry.  Indeed, actual power grids
exhibit complex behaviors not necessarily captured by small-scale
cases \cite{Josz2016_ACOPF_PegaseRTE}.

\section{Overview of the Proposed Approach}
\label{sec:overview}

The paper addresses the shortcomings in current literature by
combining learning and feasibility restoration in a single E2ELR
architecture.  Figure \ref{fig:pipeline:DCOPF} illustrates the
proposed architecture (Figure \ref{fig:pipeline:DCOPF:E2ELR}),
alongside existing architectures from previous works.  In contrast to
previous works, the proposed E2ELR uses specialized, closed-form
repair layers that allow the architecture to scale to industry-size
systems. E2ELR is also trained with self-supervised learning,
alleviating the need for labeled data and the offline solving of
numerous optimization problems. As a result, even for the largest
systems considered, the self-supervised E2ELR is trained from scratch
in under an hour, and achieves state-of-the-art performance,
outperforming other baselines by an order of magnitude. Note that
E2ELR also bridges the gap between academic DCOPF formulations and
those used in the industry, by including reserve requirements in the
ED formulation.  To the best of the authors' knowledge, this is the
first work to explicitly consider --and offer feasibility guarantees
for-- reserve requirements in the context of optimization proxies.
Moreover, the repair layers are \revision{}{guaranteed to satisfy power balance and reserve requirements for any combination of min/max generation limits (see Theorems \ref{thm:power_balance_layer} and \ref{thm:reserve_layer:guaranteed_feasibility}).}
\revision{}{This allows} to accommodate variations in
operating parameters such as min/max limits and commitment status of
generators, a key aspect of real-life systems overlooked
in existing literature.

\begin{figure}
    \centering
    \subfloat[The vanilla DNN architecture without feasibility restoration.]{
        \includegraphics[width=0.95\columnwidth]{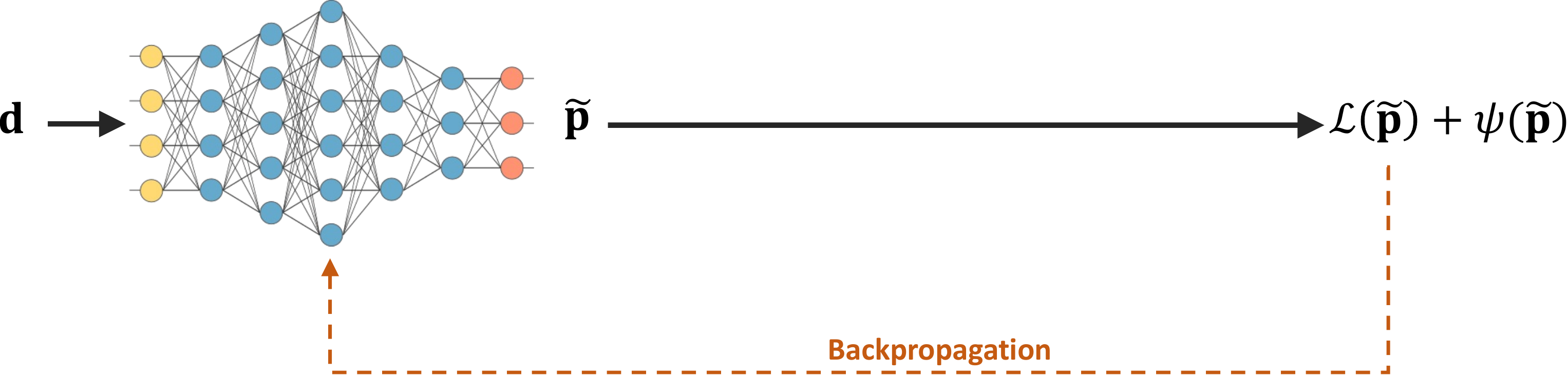}
        \label{fig:pipeline:DCOPF:DNN}
    }\\
    \subfloat[\revision{blue}{The DeepOPF architecture \cite{pan2020deepopf}. The output $\mathbf{\hat{p}}$ may violate inequality constraints.}]{
        \includegraphics[width=0.95\columnwidth]{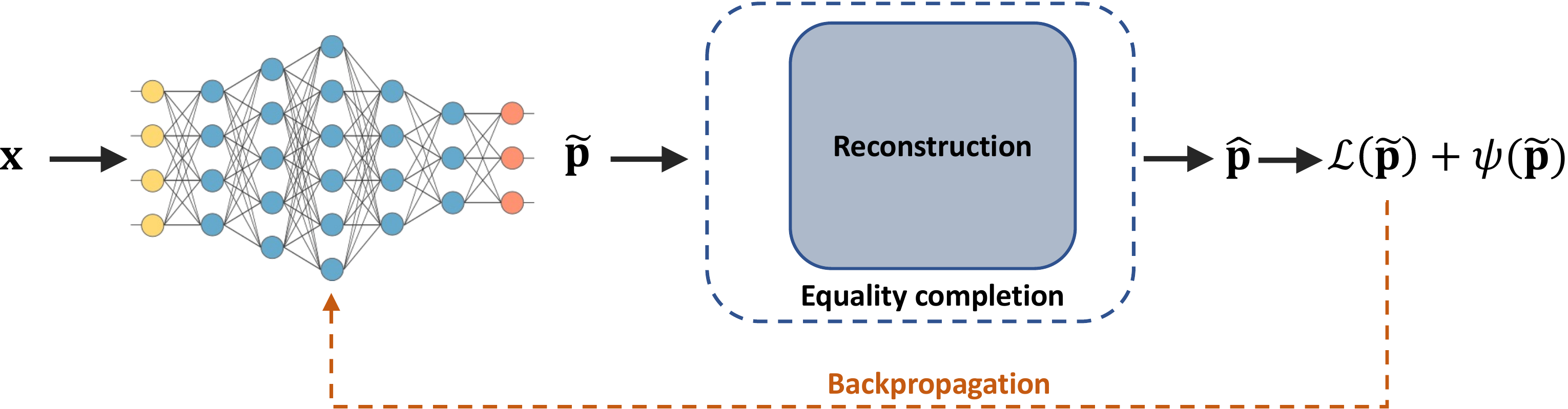}
        \label{fig:pipeline:DCOPF:DeepOPF}
    }\\
    \subfloat[The DC3 architecture with unrolled gradient \cite{donti2021dc3}. The output $\mathbf{\tilde{\tilde{p}}}$ may violate inequality constraints.]{
        \includegraphics[width=0.95\columnwidth]{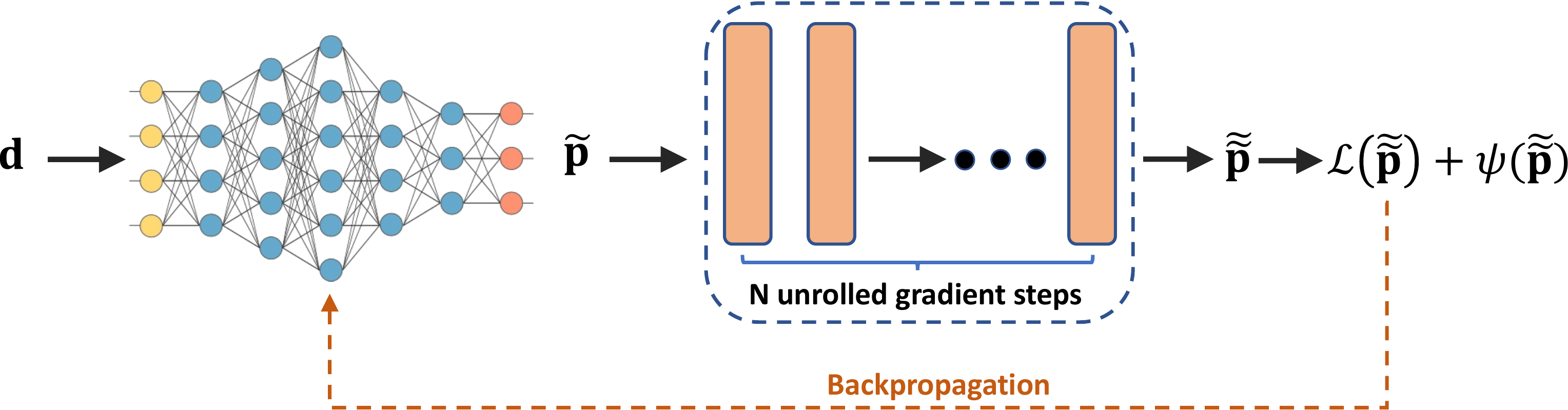}
        \label{fig:pipeline:DCOPF:DC3}
    }\\
    \subfloat[{The LOOP-LC architecture \cite{Li2022_LOOP-LC}. The DNN outputs a latent vector $\mathbf{z} \in [0,1]^{n}$,  which is mapped onto a feasible $\mathbf{\hat{p}}$ via a gauge mapping.}]{
        \includegraphics[width=0.95\columnwidth]{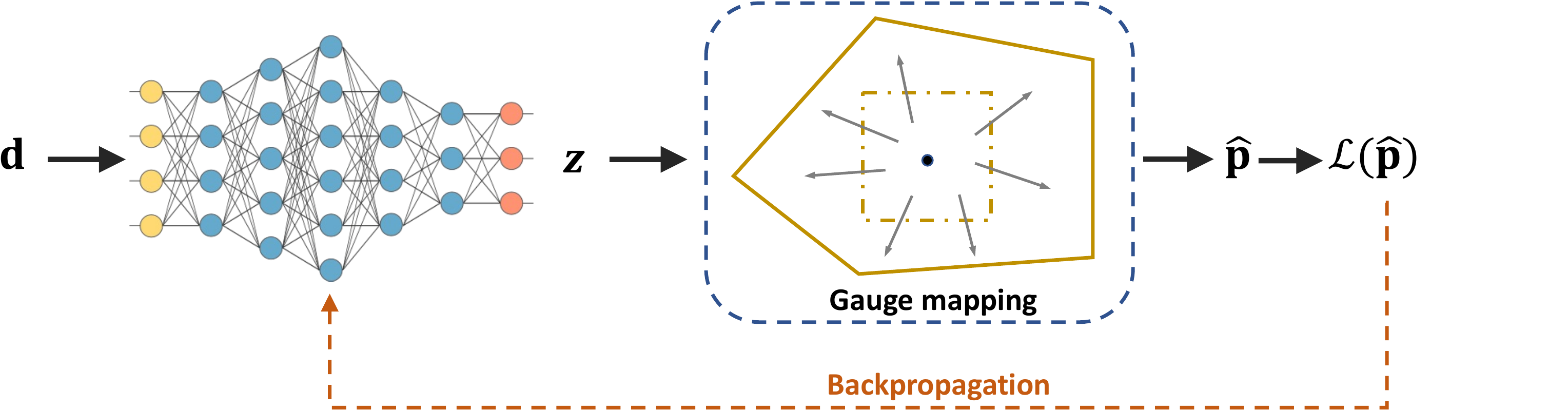}
        \label{fig:pipeline:DCOPF:LOOP}
    }\\
    \subfloat[The proposed end-to-end feasible architecture. The output $\mathbf{\hat{p}}$ satisfies all hard constraints.]{
        \includegraphics[width=0.95\columnwidth]{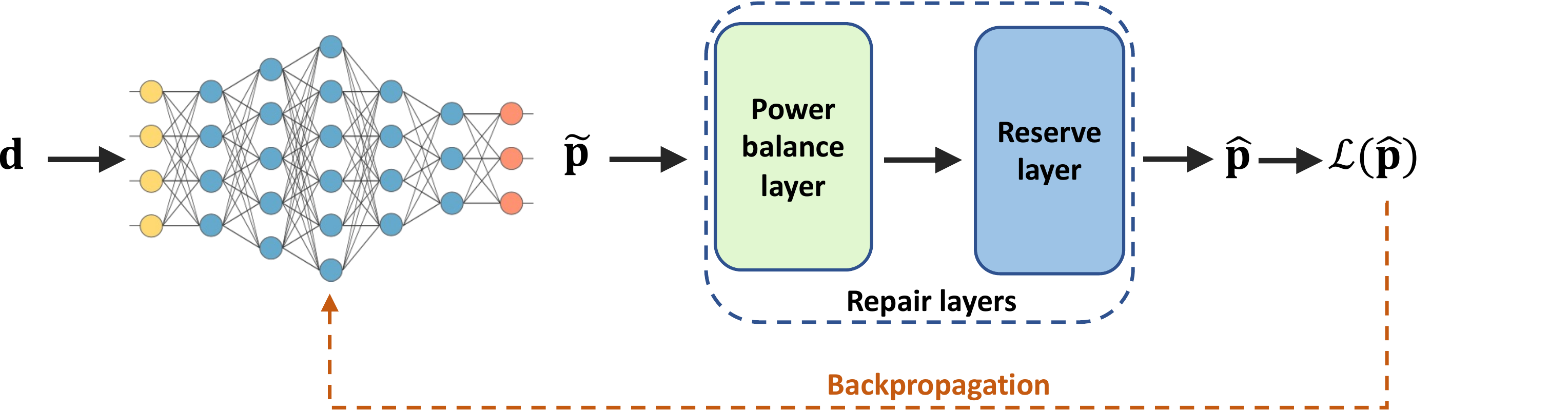}
        \label{fig:pipeline:DCOPF:E2ELR}
    }
    \caption{Optimization proxy architectures for DCOPF}
    \label{fig:pipeline:DCOPF}
\end{figure}

\section{End-to-end feasible proxies for DCOPF}
\label{sec:layers}


This section presents the Economic Dispatch (ED) formulation
considered in the paper, and introduces new repair layers for power
balance and reserve requirement constraints.  The repair layers are
computationally efficient and differentiable: they can be implemented
in standard machine-learning libraries, enabling end-to-end feasible
optimization proxies.

\subsection{Problem Formulation}
\label{sec:DCOPF}
    
The paper considers an ED formulation with reserve requirements. It
is modeled as a linear program of the form
    \begin{subequations}
    \label{eq:DCOPF}
    \begin{align}
        \min_{\pg, \res, \xith} \quad & c(\pg) + \Mth \| \xith \|_{1}\\
        \text{s.t.} \quad
            & \mathbf{e}^{\top} \pg = \mathbf{e}^{\top} \pd, \label{eq:DCOPF:power_balance}\\
            & \mathbf{e}^{\top} \res \geq R, \label{eq:DCOPF:reserve_requirements}\\
            & \pg + \res \leq \pgmax,  \label{eq:DCOPF:eco_max}\\
            & \mathbf{0} \leq \pg \leq \pgmax, \label{eq:DCOPF:dispatch_bounds}\\
            & \mathbf{0} \leq \res \leq \resmax, \label{eq:DCOPF:reserve_bounds}\\
            & \mathbf{\ubar{f}} -\xith \leq \Phi (\pg - \pd)  \leq \mathbf{\bar{f}} +\xith, \label{eq:DCOPF:PTDF}\\
            & \xith \geq \mathbf{0}.  \label{eq:DCOPF:thermal_slack_positive}
    \end{align}
    \end{subequations}

Constraints \eqref{eq:DCOPF:power_balance} and
\eqref{eq:DCOPF:reserve_requirements} are the global power balance and
minimum reserve requirement constraints, respectively.  Constraints
\eqref{eq:DCOPF:eco_max} ensure that each generator reserves can be
deployed without violating their maximum capacities.  Constraints
\eqref{eq:DCOPF:dispatch_bounds} and \eqref{eq:DCOPF:reserve_bounds}
enforce minimum and maximum limits on each generator energy and
reserve dispatch.  Without loss of generality, the paper assumes that
each bus has exactly one generator, each generator minimum output is
zero, and $\bar{r}_{g} \leq \bar{p}_{g}, \forall g$. Constraints
\eqref{eq:DCOPF:PTDF} express the thermal constraints on each branch using
a Power Transfer Distribution Factor (PTDF) representation.  In this
paper, the thermal constraints are soft constraints, i.e., they can be
violated but doing so incurs a (high) cost.  This is modeled via
artificial slack variables $\xith$ which are penalized in the
objective.  Treating thermal constraints as soft is in line with
economic dispatch formulations used by system operators to clear
electricity markets in the US \cite{Ma2009_MISO_SCED,BPM002_D}.  The
PTDF-based formulation is also the state-of-the-art approach used in
industry \cite{Ma2009_MISO_SCED,Holzer2022_MISO_SFT}.  In typical
operations, only a small number of these constraints are active at the
optimum.  Therefore, efficient implementations add thermal constraints
\eqref{eq:DCOPF:PTDF} lazily.

The hard constraints in Problem \eqref{eq:DCOPF} are the bounds on
energy and reserve dispatch, the maximum output, the power balance
\eqref{eq:DCOPF:power_balance} and the reserve requirements
\eqref{eq:DCOPF:reserve_requirements}.  Note that bounds on individual
variables can easily be enforced in a DNN architecture e.g., via
clamping or sigmoid activation. However, it is not trivial to
\emph{simultaneously} satisfy variable bounds, power balance and
reserve requirements.  To address this issue, the rest of this section
introduces new, computationally efficient repair layers.

\subsection{The Power Balance Repair Layer}
\label{sec:layers:power_balance}

    \newcommand{\hypercube}{\mathcal{H}}
    \newcommand{\hypersimplex}[1]{\mathcal{S}_{#1}}
    \newcommand{\etaup}{\eta^{\uparrow}}
    \newcommand{\etadn}{\eta^{\downarrow}}

The proposed power balance repair layer takes as input an initial
dispatch vector $\pg$, which is assumed to satisfy the min/max
generation bounds \eqref{eq:DCOPF:dispatch_bounds}, and outputs a
dispatch vector $\tilde{\pg}$ that satisfies constraints
\eqref{eq:DCOPF:dispatch_bounds} and \eqref{eq:DCOPF:power_balance}.
Formally, let $D \, {=} \, \mathbf{e}^{\top} \pd$, and denote by
$\hypercube$ and $\hypersimplex{D}$ the following hypercube and
hypersimplex
    \begin{align}
        \label{eq:layers:hypercube}
        \hypercube &= \left\{
            \pg \in \mathbb{R}^{n}
        \ \middle| \ 
            \mathbf{0} \leq \pg \leq \mathbf{\bar{p}}
        \right\},\\
        \label{eq:layers:hypersimplex}
        \hypersimplex{D} &= \left\{
            \pg \in \mathbb{R}^{n}
        \ \middle| \ 
            \mathbf{0} \leq \pg \leq \mathbf{\bar{p}}, \ 
            \mathbf{e}^{\top} \pg = D
        \right\}.
    \end{align}

\noindent
Note that $\hypercube$ is the feasible set of constraints
\eqref{eq:DCOPF:dispatch_bounds}, while $\hypersimplex{D}$ is the
feasible set of constraints \eqref{eq:DCOPF:power_balance} and
\eqref{eq:DCOPF:dispatch_bounds}. The proposed power balance repair
layer, denoted by $\mathcal{P}$, is given by
    \begin{align}
        \label{eq:power_balance_layer}
        \begin{array}{rl}
            \mathcal{P}(\pg) & = \left\{
            \begin{array}{ll}
                (1 - \etaup) \pg + \etaup \mathbf{\bar{p}} &  \text{if } \mathbf{e}^{\top} \pg < D\\
                (1 - \etadn) \pg + \etadn \mathbf{0}  &  \text{if } \mathbf{e}^{\top} \pg \geq D
            \end{array}
        \right.
        \end{array}
    \end{align}
    where $\pg \, {\in} \, \mathcal{H}$, and $\etaup$, $\etadn$ are defined as follows:
    \begin{align}
        \label{eq:power_balance_layer:ratios}
        \etaup   &= \frac{\mathbf{e}^{\top}\pd - \mathbf{e}^{\top} \pg}{\mathbf{e}^{\top} \mathbf{\bar{p}} - \mathbf{e}^{\top}\pg}, &
        \etadn &= \frac{\mathbf{e}^{\top}\pg - \mathbf{e}^{\top}\pd}{\mathbf{e}^{\top} \pg - \mathbf{e}^{\top} \mathbf{0}}.
    \end{align}
    Theorem \ref{thm:power_balance_layer} below shows that $\mathcal{P}$ is well-defined.

\begin{theorem}
    \label{thm:power_balance_layer}
    Assume that $0 < \mathbf{e}^{\top} \pd = D < \mathbf{e}^{\top} \pgmax$ and $\pg \in \hypercube$.\\
    Then, $\mathcal{P}(\pg) \in \hypersimplex{D}$.
\end{theorem}
\begin{proof}
Let $\mathbf{\tilde{p}} \, {=} \, \mathcal{P}(\pg)$, and assume $\mathbf{e}^{\top} \pg \, {<} \, D$; the case $\mathbf{e}^{\top} \pg \, {\geq} \, D$ is treated similary.
It follows that $\etaup \, {\in} \, [0, 1]$, i.e., $\mathbf{\tilde{p}}$ is a convex
combination of $\pg$ and $\pgmax$.
Thus, $\mathbf{\tilde{p}} \in
\hypercube$.  Then,
    \begin{align*}
        \mathbf{e}^{\top}\mathbf{\tilde{p}}
        &= (1 - \etaup) \mathbf{e}^{\top}\pg + \etaup \mathbf{e}^{\top}\bar{\pg}\\
        &= \etaup(\mathbf{e}^{\top}\bar{\pg} - \mathbf{e}^{\top}\pg) + \mathbf{e}^{\top} \pg\\
        &= \frac{D - \mathbf{e}^{\top} \pg}{\mathbf{e}^{\top} \mathbf{\bar{p}} - \mathbf{e}^{\top}\pg} (\mathbf{e}^{\top}\bar{\pg} - \mathbf{e}^{\top}\pg) + \mathbf{e}^{\top} \pg\\
        &= D - \mathbf{e}^{\top}\pg + \mathbf{e}^{\top} \pg = D.
    \end{align*}
    Thus, $\mathbf{\tilde{p}}$ satisfies the power balance and $\mathbf{\tilde{p}} \in \hypersimplex{D}$.
\end{proof}

Note that feasible predictions are not modified: if $\pg \in
\hypersimplex{D}$, then $\mathcal{P}(\pg) = \pg$. The edge cases not
covered by Theorem \ref{thm:power_balance_layer} are handled as
follows.  When $D \leq 0$ (resp. $D \geq \mathbf{e}^{\top} \pgmax$),
each generator is set to its lower (resp. upper) bound; this can be
achieved by clamping $\etaup$ (resp. $\etadn$) to $[0, 1]$. 
If these inequalities are strict, Problem \eqref{eq:DCOPF} is
trivially infeasible, and $\mathcal{P}(\pg)$ is the solution that
minimizes power balance violations.

The power balance layer is illustrated in Figure
\ref{fig:feasibility_layer:hypersimplex}, for a two-generator system.
The layer has an intuitive interpretation as a proportional response
mechanism.  Indeed, if the initial dispatch $\pg$ has an energy
shortage, i.e., $\mathbf{e}^{\top} \pg \, {<} \, D$, the output of
each generator is increased by a fraction $\etaup$ of its upwards
headroom.  Likewise, if the initial dispatch has an energy surplus,
i.e., $\mathbf{e}^{\top} \pg \, {>} \, D$, the output of each
generator is decreased by a fraction $\etadn$ of it downwards
headroom.

\begin{figure}[!t]
        \centering
        \resizebox{0.48\columnwidth}{!}{
        \begin{tikzpicture}[x=2cm,y=2cm]
            \draw[black,-stealth] (-0.1, 0.0) -- (1.25, 0.0);
            \node[below] at (1.25, 0.0) {\footnotesize $p_{1}$};
            \draw[black,-stealth] (0.0, -0.1) -- (0.0, 1.25);
            \node[left] at (0.0, 1.25) {\footnotesize $p_{2}$};
            
            \draw[black] (0,0) -- (0, 1) -- (1,1) -- (1,0) -- cycle;
            
            
            \draw[red,line width=1pt] (0.1, 1.1) -- (1.1, 0.1);
            \node[right, color=red] at (0.1, 1.2) {\footnotesize $\mathbf{e}^{\top}\pg = D$};
            
            \node[draw=blue, fill=blue, inner sep=0pt,minimum size=3pt] (p_bar) at (1.00, 1.00) {};
            \node[diamond,draw=blue, fill=blue, inner sep=0pt,minimum size=3pt] (p_hat) at (0.30, 0.30) {};
            \node[left] at (0.3, 0.3) {\footnotesize $\pg$};
            \node[circle,draw=blue, fill=blue, inner sep=0pt,minimum size=3pt] (p_til) at (0.6, 0.6) {};
            \node[right] at (0.65, 0.6) {\footnotesize $\mathbf{{\tilde{p}}}$};
            
            \draw[black,-stealth] (p_hat) -- (p_til);
            \draw[black,densely dotted] (p_til) -- (p_bar);
        \end{tikzpicture}}
        \hfill
        \resizebox{0.48\columnwidth}{!}{
        \begin{tikzpicture}[x=2cm,y=2cm]
            \draw[black,-stealth] (-0.1, 0.0) -- (1.25, 0.0);
            \node[below] at (1.25, 0.0) {\footnotesize $p_{1}$};
            \draw[black,-stealth] (0.0, -0.1) -- (0.0, 1.25);
            \node[left] at (0.0, 1.25) {\footnotesize $p_{2}$};
            
            \draw[black] (0,0) -- (0, 1) -- (1,1) -- (1,0) -- cycle;
            
            
            \draw[red,line width=1pt] (0.1, 1.1) -- (1.1, 0.1);
            \node[right, color=red] at (0.1, 1.2) {\footnotesize $\mathbf{e}^{\top}\pg = D$};
            
            \node[draw=blue, fill=blue, inner sep=0pt,minimum size=3pt] (p_bar) at (0.00, 0.00) {};
            \node[diamond,draw=blue, fill=blue, inner sep=0pt,minimum size=3pt] (p_hat) at (0.80, 0.80) {};
            \node[right] at (0.8, 0.8) {\footnotesize $\pg$};
            \node[circle,draw=blue, fill=blue, inner sep=0pt,minimum size=3pt] (p_til) at (0.6, 0.6) {};
            \node[left] at (0.58, 0.58) {\footnotesize $\mathbf{{\tilde{p}}}$};
            
            \draw[black,-stealth] (p_hat) -- (p_til);
            \draw[black,densely dotted] (p_til) -- (p_bar);
        \end{tikzpicture}}
        \hfill\\
        \caption{%
            Illustration of the power balance layer with input $\pg$ and output $\mathbf{\tilde{p}}$.
            Left: $\mathbf{e}^{\top}\pg \, {<} \, D$ (energy shortage) and generators' dispatches are increased.
            Right: $\mathbf{e}^{\top}\pg \, {>} \, D$ (energy surplus) and generators' dispatches are decreased.
        }
        \label{fig:feasibility_layer:hypersimplex}
\end{figure}
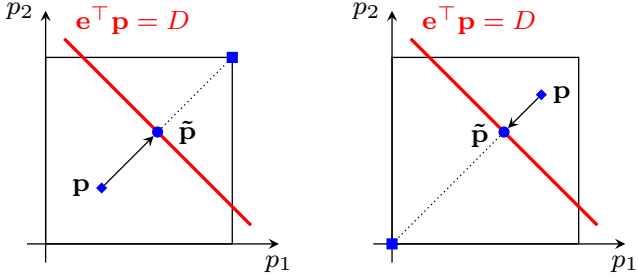

Note that $\hypersimplex{D}$ is defined by the combination of bound
constraints and one equality constraint.  Other works such as
\cite{donti2021dc3,Li2022_LOOP-LC} handle the latter via equality
completion.  While this approach satisfies the equality constraint by
design, the recovered solution is not guaranteed to satisfy min/max
bounds, and may fail to do so in general.  In contrast, under the only
assumption that $\pg \in \hypercube$, the proposed layer
\eqref{eq:power_balance_layer} jointly enforces \emph{both}
constraints, thus alleviating the need for the gradient unrolling of
\cite{donti2021dc3} or gauge mapping of \cite{Li2022_LOOP-LC}.
Finally, the proposed layer generalizes to hypersimplices of the form
    \begin{align}
        \label{eq:power_balance_layer:hypersimplex:general}
        \left\{
            x \in \mathbb{R}^{n}
        \ \middle| \ 
            l \leq x \leq u, a^{T}x = b
        \right\},
    \end{align}
    where $a \in \mathbb{R}^{n}, b \in \mathbb{R}$ and $l \leq u \in \mathbb{R}^{n}$ are finite bounds.

\subsection{The Reserve Repair Layer}
\label{sec:layers:reserve}

This section presents the proposed reserve feasibility layer, which
ensures feasibility with respect to constraints
\eqref{eq:DCOPF:reserve_requirements}, \eqref{eq:DCOPF:eco_max}, and
\eqref{eq:DCOPF:reserve_bounds}.  The approach first builds a compact
representation of these constraints by projecting out the reserve
variables $\res$.  This makes it possible to consider only the $\pg$
variables, which in turn enables a computationally efficient and
interpretable feasibility restoration. Let $\pg \in \hypersimplex{D}$
be fixed, and consider the problem of maximizing total reserves, which
reads
    \begin{subequations}
    \label{eq:ReserveRecovery}
    \begin{align}
        \max_{\mathbf{r}} \quad & \mathbf{e}^{\top} \mathbf{r} \label{eq:ReserveRecovery:obj}\\
        \text{s.t.} \quad & \mathbf{r} \leq \bar{\pg} - \pg, \label{eq:ReserveRecovery:eco_max}\\
        & 0 \leq \mathbf{r} \leq \bar{\mathbf{r}}. \label{eq:ReserveRecovery:reserve_bounds}
    \end{align}
    \end{subequations}
    Since $\pg$ is fixed, constraints \eqref{eq:ReserveRecovery:eco_max}--\eqref{eq:ReserveRecovery:reserve_bounds} reduce to simple variable bounds on the $\mathbf{r}$ variables.
    It then immediately follows that the optimal solution to Problem \eqref{eq:ReserveRecovery} is given by
    \begin{align}
        \label{eq:reserve_recovery:optimal_reserves}
        r^{*}_{g} = \min\{\bar{r}_{g}, \bar{p}_{g} - p_{g} \}, \ \forall g.
    \end{align}
    This observation is used to project out the reserve variables as stated in Lemma \ref{lem:reserve_layer:compact_representation} below.

\begin{lemma}
\label{lem:reserve_layer:compact_representation}
Let $\pg \in \hypersimplex{D}$. There exists reserves $\mathbf{r}$ such that $(\pg, \mathbf{r})$ is feasible for Problem \eqref{eq:DCOPF} if and only if
\begin{align}
\label{eq:reserves:sufficient}
\sum_{g} \min\{\bar{r}_{g}, \bar{p}_{g} - p_{g} \} \geq R.
\end{align}
\end{lemma}
\begin{proof}
The proof follows from the fact that $(\pg, \mathbf{r})$ is feasible
for Problem \eqref{eq:DCOPF} if and only if Problem
\eqref{eq:ReserveRecovery} has an objective value not smaller than
$R$. Substituting the optimal solution given in
Eq. \eqref{eq:reserve_recovery:optimal_reserves}, this last statement
is exactly equivalent to $\sum_{g} \min\{\bar{r}_{g}, \bar{p}_{g} -
p_{g} \} \geq R$.
\end{proof}
    
The proposed reserve repair layer builds on the power balance repair
layer of Section \ref{sec:layers:power_balance}, and on the compact
formulation of Eq. \eqref{eq:reserves:sufficient}.  Namely, it takes
as input $\pg \, {\in} \, \hypersimplex{D}$, and outputs
$\mathcal{R}(\pg) \, {\in} \, \hypersimplex{D}$ that satisfies
Eq. \eqref{eq:reserves:sufficient}.  Given $\mathcal{R}(\pg)$, reserve
variables can be recovered in $O(|\mathcal{G}|)$ time using
Eq. \eqref{eq:reserve_recovery:optimal_reserves}.

The reserve repair layer is presented in Algorithm
\ref{alg:FFR:reserves}.  First, a tentative reserve allocation is
computed using Eq. \ref{eq:reserve_recovery:optimal_reserves}, and the
corresponding reserve shortage $\Delta_{R}$ is computed.  Then,
generators are split into two groups $G^{\uparrow}$ and
$\mathcal{G}^{\downarrow}$.  Generators in $\mathcal{G}^{\uparrow}$
are those for which constraint \eqref{eq:DCOPF:reserve_bounds} is
active: their dispatch can be increased without having to reduce their
reserves.  Generators in $\mathcal{G}^{\downarrow}$ are those for
which constraint \eqref{eq:DCOPF:eco_max} is active: one must reduce
their energy dispatch to increase their reserves.  Then, the algorithm
computes the maximum possible increase ($\Delta^{\uparrow}$) and
decrease ($\Delta^{\downarrow})$ in energy dispatch for the two
groups.  Finally, each generator energy dispatch is increased
(resp. decreased) proportionately to its increase (resp. decrease)
potential so as to meet total reserve requirements.  The total
increase in energy dispatch is equal to the total decrease, so that
power balance is always maintained.
    
\begin{algorithm}[!t]
        \caption{Reserve Repair Layer}
        \label{alg:FFR:reserves}
        \begin{algorithmic}[1]
            \REQUIRE Initial prediction $\pg \in \hypersimplex{D}$, maximum limits $\mathbf{\bar{p}}, \mathbf{\bar{r}}$, reserve requirement $R$
            \STATE $\Delta_{R} \gets R - \sum_{g} \min \{ \bar{r}_{g}, \bar{p}_{g} - p_{g} \}$
            \STATE $\mathcal{G}^{\uparrow} \gets \left\{ g \ \middle| \ p_{g} \leq \bar{p}_{g} - \bar{r}_{g} \right\}$
            \STATE $\mathcal{G}^{\downarrow} \gets \left\{ g \ \middle| \  p_{g} > \bar{p}_{g} - \bar{r}_{g} \right\}$
            \STATE $\Delta^{\uparrow} \gets \sum_{g \in \mathcal{G}^{\uparrow}} (\bar{p}_{g} - \bar{r}_{g}) - p_{g}$
            \STATE $\Delta^{\downarrow} \gets \sum_{g \in \mathcal{G}^{\downarrow}} p_{g} - (\bar{p}_{g} - \bar{r}_{g})$
            \STATE $\Delta \gets \max(0, \min(\Delta_{R}, \Delta^{\uparrow}, \Delta^{\downarrow}))$
            \STATE $\alpha^{\uparrow} \gets \Delta / \Delta^{\uparrow}$, $\alpha^{\downarrow} \gets \Delta / \Delta^{\downarrow}$, \label{alg:FFR:proportional_response_weigths}
            \STATE Energy dispatch adjustment
                \begin{align*}
                    \tilde{p}_{g} &= \left\{
                        \begin{array}{ll}
                            (1 - \alpha^{\uparrow}) p_{g} + \alpha^{\uparrow} (\bar{p}_{g} - \bar{r}_{g}) & \forall g \in \mathcal{G}^{\uparrow} \\
                            (1 - \alpha^{\downarrow}) p_{g} + \alpha^{\downarrow} (\bar{p}_{g} - \bar{r}_{g}) & \forall g \in \mathcal{G}^{\downarrow}
                        \end{array}
                    \right.
                \end{align*}
                \label{alg:FFR:proportional_response}
            \STATE \textbf{return} $\mathcal{R}(\pg) = \mathbf{\tilde{p}}$
        \end{algorithmic}
    \end{algorithm}

The reserve feasibility recovery is illustrated in Figure
\ref{fig:feasibility_recovery:reserve} for a two-generator system.
While it is easy to verify that $\mathcal{R}(\pg) \in
\hypersimplex{D}$, it is less clear whether $\mathcal{R}(\pg)$
satisfies Eq. \eqref{eq:reserves:sufficient}.  Theorem
\ref{thm:reserve_layer:guaranteed_feasibility} provides the
theoretical guarantee that either $\mathcal{R}(\pg)$ satisfies
Eq. \eqref{eq:reserves:sufficient}, or no feasible solution to Problem
\eqref{eq:DCOPF} exists.  Because of space limitations, the proof is
given in \cite{Chen2023_E2ELR_arxiv}.

\begin{theorem}
        \label{thm:reserve_layer:guaranteed_feasibility}
        Let $\pg \in \hypersimplex{D}$.
        Then, $\mathcal{R}(\pg) \in \hypersimplex{D}$.
        Furthermore, $\mathcal{R}(\pg)$ satisfies Eq. \eqref{eq:reserves:sufficient} if and only if Problem \eqref{eq:DCOPF} is feasible.
\end{theorem}

Theorem \ref{thm:reserve_layer:guaranteed_feasibility} also provides a
fast proof of (in)feasibility for Problem \eqref{eq:DCOPF}: it
suffices to evaluate $\mathcal{R}(\pg)$ for any $\pg \in
\hypersimplex{D}$ and check Eq. \eqref{eq:reserves:sufficient}.  This
can have applications beyond optimization proxies, e.g., to quickly
evaluate a large number of scenarios for potential reserve violations.
\revision{}{Finally, note that the results of Theorems \ref{thm:power_balance_layer} and \ref{thm:reserve_layer:guaranteed_feasibility} hold for any value of the maximum limits $\pgmax, \bar{\mathbf{r}}$ and reserve requirement $R$.}
    
    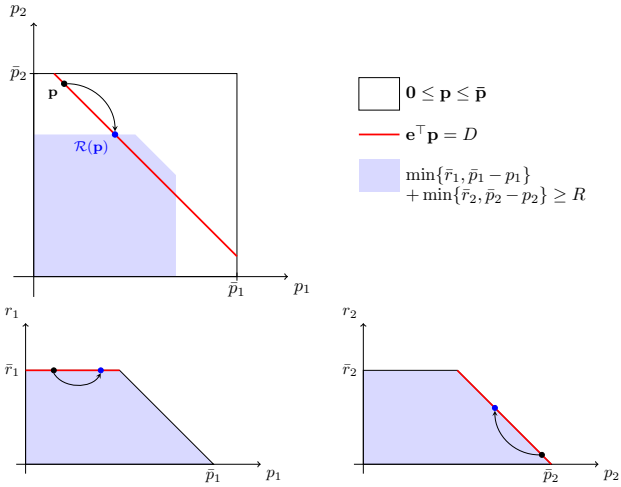
\begin{figure}[!t]
        \centering
        \resizebox{0.9\columnwidth}{!}{
        \begin{tikzpicture}[x=4cm,y=4cm]
            \draw[black,->] (-0.1, 0.0) -- (1.25, 0.0);
            \node[below right] at (1.25, 0.0) {$p_{1}$};
                \draw[black] (1, -0.02) -- (1, 0.02);
                \node[below] at (1, 0.0) {$\bar{p}_{1}$};
            
            \draw[black,->] (0.0, -0.1) -- (0.0, 1.25);
            \node[above left] at (0.0, 1.25) {$p_{2}$};
                \draw[black] (-0.02, 1) -- (+0.02, 1);
                \node[left] at (0, 1) {$\bar{p}_{2}$};
                
            \fill[blue!50!white, fill opacity=0.3] (0,0) -- (0,0.7) -- (0.5, 0.7) -- (0.7, 0.5) -- (0.7, 0.0) -- cycle;
            
            \draw[black] (0,0) -- (0, 1) -- (1,1) -- (1,0) -- cycle;
            
            \draw[red,line width=1pt] (0.1, 1.0) -- (1.0, 0.1);
            
            \node[circle,draw=black, fill=black, inner sep=0pt,minimum size=3pt] (phat_a) at (0.15, 0.95) {};
            \node[below left] at (0.15, 0.95) {\footnotesize $\pg$};
            
            \node[circle,draw=blue, fill=blue, inner sep=0pt,minimum size=3pt] (pfeas1) at (0.4, 0.7) {};
            \node[below left, blue] at (0.4, 0.7) {\footnotesize $\mathcal{R}(\pg)$};
            
            \draw[-stealth] (phat_a.east) to [out=0,in=90] (pfeas1.north);
            
                \draw[black] (1.6, 0.98) -- (1.8, 0.98) -- (1.8, 0.82) -- (1.6, 0.82) -- cycle;
                \node[right] at (1.8, 0.9) {$\mathbf{0} \leq \pg \leq \mathbf{\bar{p}}$};
                \draw[red, line width=1pt] (1.6, 0.7) -- (1.8, 0.7);
                \node[right] at (1.8, 0.7) {$\mathbf{e}^{\top} \pg = D$};
                \fill[blue!50!white, fill opacity=0.3] (1.6, 0.58) -- (1.8, 0.58) -- (1.8, 0.42) -- (1.6, 0.42) -- cycle;
                \node[right] at (1.8, 0.5) {$\min\{\bar{r}_{1}, \bar{p}_{1} \, {-} \, p_{1}\}$};
                \node[right] at (1.8, 0.4) {$+\min\{\bar{r}_{2}, \bar{p}_{2} \, {-} \, p_{2}\} \geq R$};
        \end{tikzpicture}
        }\\
        \hfill
        \resizebox{0.45\columnwidth}{!}{
            \begin{tikzpicture}[x=4cm,y=4cm]
                \draw[black,->] (-0.05, 0) -- (1.25, 0);
                \node[below right] at (1.25, 0) {$p_{1}$};
                \node[below] at (1, 0) {$\bar{p}_{1}$};
                
                \draw[black,->] (0, -0.05) -- (0, 0.75);
                \node[above left] at (0, 0.75) {$r_{1}$};
                \node[left] at (0, 0.5) {$\bar{r}_{1}$};
                
                \fill[blue!50!white, fill opacity=0.3] (0,0) -- (0, 0.5) -- (0.5, 0.5) -- (1,0) -- cycle;
                \draw[black] (0,0) -- (0, 0.5) -- (0.5,0.5) -- (1,0) -- cycle;
                \draw[red, line width=1pt] (0.0, 0.5) -- (0.5, 0.5);
                
                \node[circle,draw=black, fill=black, inner sep=0pt,minimum size=3pt] (ppred) at (0.15, 0.5) {};
                \node[circle,draw=blue, fill=blue, inner sep=0pt,minimum size=3pt] (pfeas) at (0.4, 0.5) {};
                \draw[-stealth] (ppred.south) to [out=-60,in=-120] (pfeas.south);
            \end{tikzpicture}
            }
        \hfill
        \resizebox{0.45\columnwidth}{!}{
            \begin{tikzpicture}[x=4cm,y=4cm]
                \draw[black,->] (-0.05, 0) -- (1.25, 0);
                \node[below right] at (1.25, 0) {$p_{2}$};
                \node[below] at (1, 0) {$\bar{p}_{2}$};
                
                \draw[black,->] (0, -0.05) -- (0, 0.75);
                \node[above left] at (0, 0.75) {$r_{2}$};
                \node[left] at (0, 0.5) {$\bar{r}_{2}$};
                
                \fill[blue!50!white, fill opacity=0.3] (0,0) -- (0, 0.5) -- (0.5, 0.5) -- (1,0) -- cycle;
                \draw[black] (0,0) -- (0, 0.5) -- (0.5, 0.5) -- (1,0) -- cycle;
                \draw[red, line width=1pt] (0.5, 0.5) -- (1.0, 0.0);
                
                \node[circle,draw=black, fill=black, inner sep=0pt,minimum size=3pt] (ppred) at (0.95, 0.05) {};
                \node[circle,draw=blue, fill=blue, inner sep=0pt,minimum size=3pt] (pfeas) at (0.7, 0.3) {};
                \draw[-stealth] (ppred.west) to [out=180,in=-90] (pfeas.south);
            \end{tikzpicture}
            }
        \caption{Illustration of the reserve feasibility layer for $\mathbf{\bar{p}} {=} (1, 1)$, $\mathbf{\bar{r}}{=}(0.5, 0.5)$, $D {=} 1.1$, $R{=}0.8$ and the initial prediction $\pg {=} (0.15, 0.95)$. The recovered feasible dispatch is $\mathbf{\tilde{p}} {=} (0.4, 0.7)$.
        Top: effect of the layer in the $(p_{1}, p_{2})$ space.
        Bottom: effect of the layer on each generator (individually). The active constraint is shown in red. Generator $1$ is in $\mathcal{G}^{\uparrow}$ and generator 2 is in $\mathcal{G}^{\downarrow}$.}
        \label{fig:feasibility_recovery:reserve}
    \end{figure}

\subsection{End-to-end Feasible Training}
\label{sec:layers:end-to-end}

The repair layers are combined with a Deep Neural Network (DNN)
architecture to provide an end-to-end feasible ML model, i.e., a
differentiable architecture that is guaranteed to output a feasible
solution to Problem \eqref{eq:DCOPF} (if and only if one exists).  The
resulting architecture is illustrated in Figure
\ref{fig:end2end:architecture_detailed}.  The proxy takes as input the
vector of nodal demand $\pd$.  The DNN architecture consists of
fully-connected layers with ReLU activation, and a final layer with
sigmoid activations to enforce bound constraints on $\pg$.  Namely,
the last layer outputs $\mathbf{z} \in [0, 1]^{n}$, and
$\mathbf{\tilde{p}} = \mathbf{z} \cdot \pgmax$ satisfies constraints
\eqref{eq:DCOPF:dispatch_bounds}.  Then, this initial prediction
$\mathbf{\tilde{p}}$ is fed to the repair layers that restore the
feasibility of the power balance and reserve requirements.  The final
prediction $\mathbf{\hat{p}}$ is feasible for Problem
\eqref{eq:DCOPF}.
    
\begin{figure}[!t]
        \centering
        \includegraphics[width=\columnwidth]{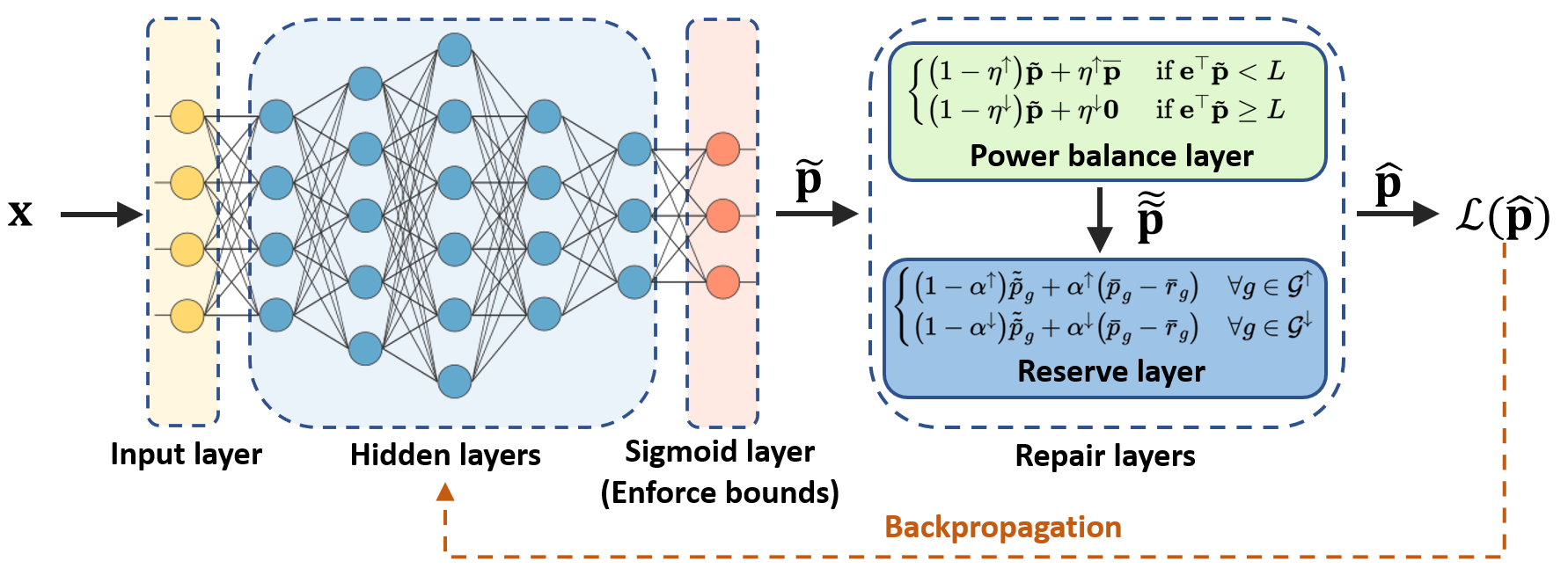}
        \caption{The Proposed End-To-End Feasible Architecture.}
        \label{fig:end2end:architecture_detailed}
    \end{figure}

The power balance and reserve feasibility layers only require
elementary arithmetic and logical operations, all of which are
supported by mainstream ML libraries like PyTorch and TensorFlow.
Therefore, it can be implemented as a layer of a generic artificial
neural network model trained with back-propagation.  Indeed, these
layers are differentiable almost everywhere with informative
(sub)gradients.  Finally, the proposed feasibility layers can be used
as a stand-alone, post-processing step to restore feasibility of any
dispatch vector that satisfies generation bounds.  This can be used
for instance to build fast heuristics with feasibility guarantees.

\section{Training methodology}
\label{sec:trainig_methodology}

This section describes the supervised learning (SL) and
self-supervised learning (SSL) approaches for training optimization
proxies, \revision{}{which are illustrated in Figure \ref{fig:training}}.
The difference between these two paradigms lies in the
choice of loss function for training, not in the model architecture.
Denote by $\x$ the input data of Problem \eqref{eq:DCOPF}, i.e.,
\begin{align*}
    \x = (c, \pd, R, \pgmax, \resmax, \Phi, \pfmin, \pfmax, \Mth),
\end{align*}
and recall, from Section \ref{sec:layers:reserve}, that it is
sufficient to predict the (optimal) value of variables $\pg$. Denote
by  $f_{\theta}$ the mapping of a DNN architecture with trainable
parameters $\theta$; given an input $\x$, $f_{\theta}(\x)$
predicts a generator dispatch \revision{}{$f_{\theta}(\x) = \hat{\pg}$}.

\begin{figure}[!t]
    \centering
    \subfloat[\revision{}{Supervised Learning: the loss function $\mathcal{L}^{SL}$ measures the distance to an optimal solution $\mathbf{p}^{*}$, which is computed offline by an optimization solver.}]{
        \includegraphics[width=0.95\columnwidth]{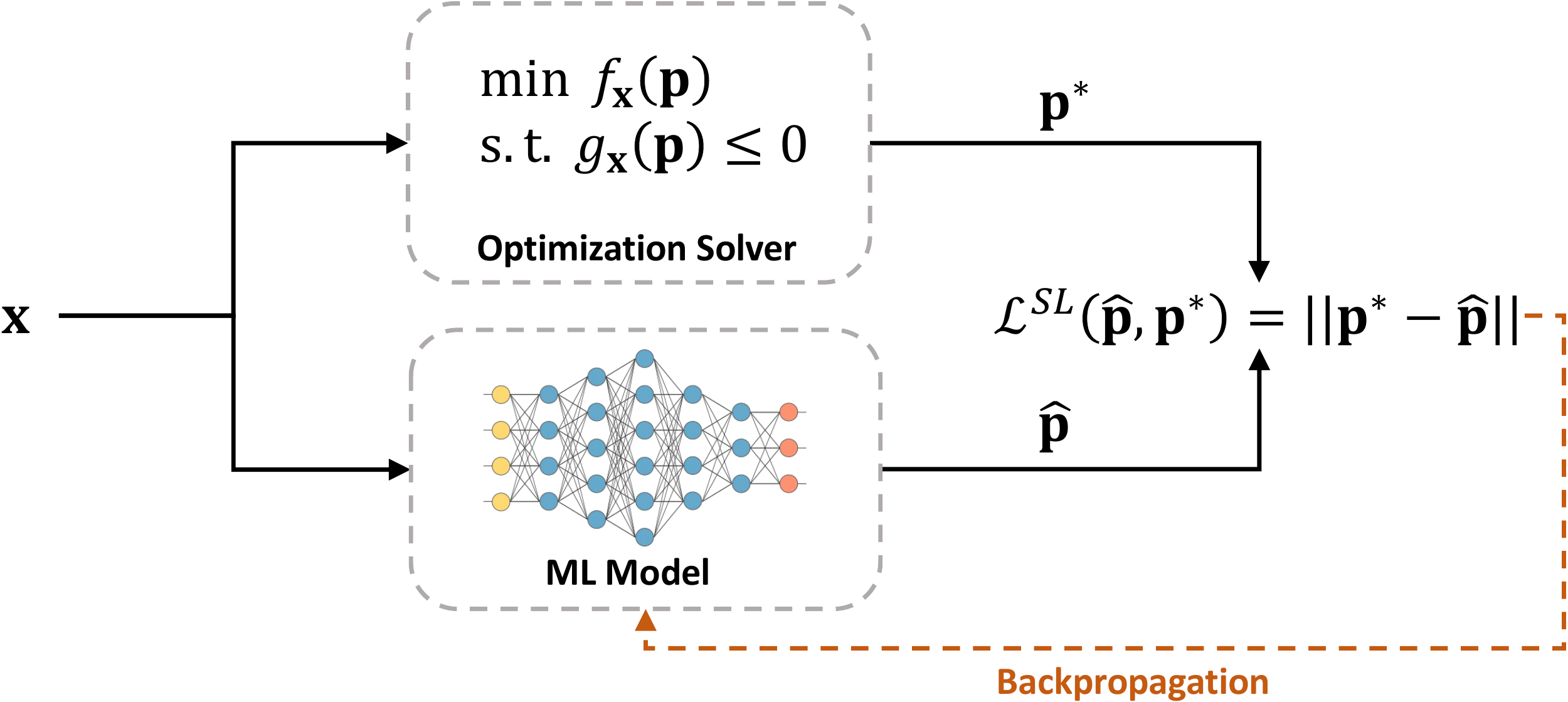}
        \label{fig:training:SL}
    }\\
    \subfloat[\revision{}{Self-Supervised Learning: the loss function $\mathcal{L}^{SSL}$ measures the objective value of the prediction. No target $\mathbf{p}^{*}$ nor optimization solver is needed.}]{
        \includegraphics[width=0.95\columnwidth]{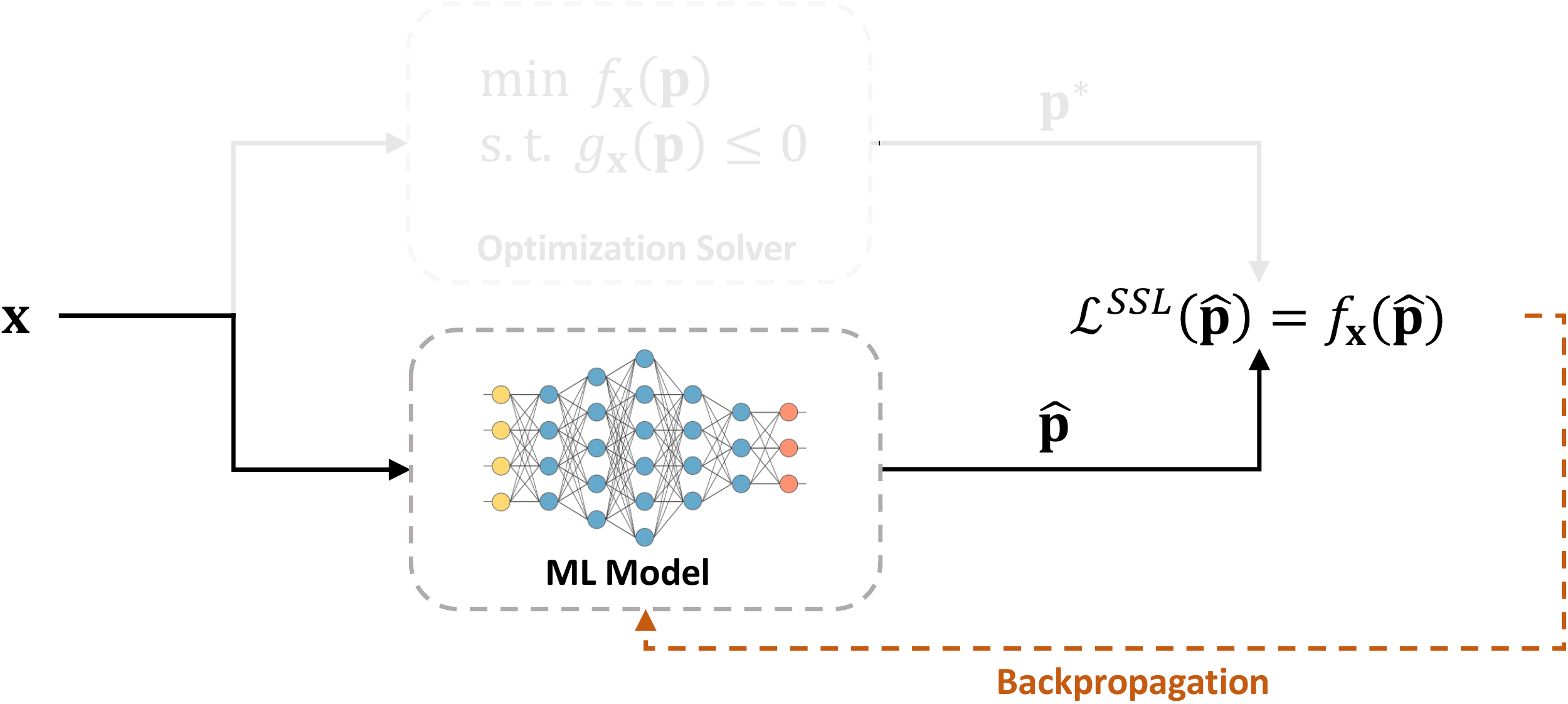}
        \label{fig:training:SSL}
    }
    \caption{\revision{}{Illustration of supervised and self-supervised learning paradigms. When the ML model is not guaranteed to produce feasible solutions, the loss function $\mathcal{L}$ can be augmented with a constraint penalization term.}}
    \label{fig:training}
\end{figure}

Consider a dataset of $N$ data points
\begin{align}
    \mathcal{D} &= \Big\{
        \big( \x^{(1)}, \pg^{(1)} \big),
        ..., 
        \big( \x^{(N)}, \pg^{(N)} \big)
    \Big\},
\end{align}
where each data point corresponds to an instance of Problem
\eqref{eq:DCOPF} and its solution, i.e., $\x^{(i)}$ and $\pg^{(i)}$
denote the input data and solution of instance $i \, {\in} \, \{1,
..., N\}$, respectively. The training of the DNN $f_{\theta}$ can be
formalized as the optimization problem
\begin{align}
    \label{eq:training}
    \theta^* = \argmin_{\theta} \quad & \frac{1}{N} \sum_{i=1}^{N} \mathcal{L} \left( \hat{\pg}^{(i)}, \pg^{(i)} \right),
\end{align}
where $\hat{\pg}^{(i)} = f_{\theta^*}(\x^{(i)})$ is the prediction for
instance $i$, and $\mathcal{L}$ denotes the loss function.
The rest of this section describes the choice of \revision{}{loss function} $\mathcal{L}$ for the SL and SSL settings.

\subsection{Supervised Learning}
\label{sec:training:supervised}

    The supervised learning loss $\mathcal{L}^{\text{SL}}$ has the form
    \begin{align}
        \label{eq:SL:loss}
        \mathcal{L}^{\text{SL}} (\hat{\pg}, \pg) &= \varphi^{\text{SL}}(\hat{\pg}, \pg) + \lambda \psi(\hat{\pg}),
    \end{align}
    where $\varphi^{\text{SL}}(\hat{\pg}, \pg)$ penalizes the distance between the predicted and the target (ground truth) solutions, and $\psi(\hat{\pg})$ penalizes constraint violations.
    The paper uses the Mean Absolute Error (MAE) on energy dispatch, i.e.,
    \begin{align}
        \label{eq:SL:loss:MAE}
        \varphi^{\text{SL}}(\hat{\pg}, \pg) &= \frac{1}{|\mathcal{G}|} \| \hat{\pg} - \pg \|_{1}.
    \end{align}
    Note that other loss functions, e.g., Mean Squared Error (MSE),
    could be used instead.  The term $\psi(\hat{\pg})$ penalizes power
    balance violations and reserve shortages as follows:
    \begin{align}
        \label{eq:SL:constraint_penalty}
        \psi(\hat{\pg}) = \Mpb |\mathbf{e}^{\top}\pd - \mathbf{e}^{\top} \hat{\pg} | + \Mres \xir(\hat{\pg}),
    \end{align}
    where $\Mpb$ and $\Mres$ are penalty coefficients, and $\xir$ denotes the reserve shortages, i.e.,
    \begin{align}
        \label{eq:SL:reserve_shortage}
        \xir(\hat{\pg}) = \max \big\{ 0, R - \sum_g \min (\bar{r}_g, \bar{p}_g - \hat{p}_g) \big\}.
    \end{align}
    The penalty term $\psi$ is set to zero for end-to-end feasible
    models. Finally, while thermal constraints are soft, preliminary
    experiments found that including thermal violations in the loss
    function yields more accurate models.
    \revision{}{This observation echoes the findings of multiple studies \cite{Ng2018_StatisticalLearningDCOPF,pan2020deepopf,fioretto2020predicting,Nellikkath2021_PINN-DCOPF,donti2021dc3}, namely, that penalizing constraint violations yields better accuracy and improves generalization.
    Indeed, the penalization of thermal violations acts as a regularizer in the loss function.}
    Therefore, the final loss
    considered in the paper is
    \begin{align}
        \label{eq:SL:loss:final}
        \mathcal{L}^{\text{SL}} (\hat{\pg}, \pg) &= \varphi^{\text{SL}}(\hat{\pg}, \pg) + \lambda \psi(\hat{\pg}) + \mu \Mth \| \xith(\hat{\pg}) \|_1,
    \end{align}
    where $\xith(\hat{\pg})$ denotes thermal violations
    (Eq. \eqref{eq:DCOPF:PTDF}).

\subsection{Self-supervised Learning}
\label{sec:training:self_supervised}

\revision{}{Self-supervised learning has been applied very recently to train optimization proxies \cite{donti2021dc3,Huang2021_DeepOPF-NGT,park2023self}.
The core aspect of SSL is that \emph{it does not require the labeled data $\mathbf{p}$}, because it directly minimizes the objective function of the original problem.
The self-supervised loss guides the training to imitate the solving of optimization instances using gradient-based algorithms, which makes it effective for optimization proxies.}

The loss function
$\mathcal{L}^{\text{SSL}}$ has the form
    \begin{align}
        \label{eq:SSL:loss}
        \revision{}{\mathcal{L}^{\text{SSL}} (\hat{\pg})} &= \varphi^{\text{SSL}}(\hat{\pg}) + \lambda \psi(\hat{\pg}),
    \end{align}
    where $\psi(\hat{\pg})$ is the same as Eq. \eqref{eq:SL:constraint_penalty}, and 
    \begin{align}
        \label{eq:SSL:objective}
        \varphi^{\text{SSL}}(\hat{\pg}) &= c(\hat{\pg}) + \Mth \xith(\hat{\pg})
    \end{align}
is the objective value of the predicted solution.
\revision{}{As mentioned above, note that $\mathcal{L}^{SLL}$ only depends on the predicted solution $\hat{\pg}$, i.e., unlike the supervised learning loss $\mathcal{L}^{SL}$ (see Eq. \eqref{eq:SL:loss:final}), it does not require any ground truth solution $\pg$.
Consequently, SSL does not require labeled data.
This, in turn, eliminates the need to solve numerous instances offline.}

\revision{}{Again, the constraint penalty term is zero when training end-to-end feasible models, because their output satisfies all hard constraints.
Thereby, the self-supervised loss is very effective since the learning focuses on optimality.}

\revision{}{For models without feasibility guarantees, the trade-off between optimality and feasibility typically makes the training very hard to stabilize for large-scale systems, which increases the learning difficulty.
Thus, such models must take care to satisfy constraints to avoid
spurious solutions.  For instance, in the ED setting, simply
minimizing total generation cost, without considering the power balance constraint \eqref{eq:DCOPF:power_balance}, yields the trivial solution $\pg \,
{=} \, \mathbf{0}$.  This highlights the importance of ensuring
feasibility, which is the core advantage of the proposed end-to-end
feasible architecture.}

\section{Experiment settings}
\label{sec:experiments}

\newcommand{\ieeeSmall}{\texttt{ieee300}}
\newcommand{\pegaseSmall}{\texttt{pegase1k}}
\newcommand{\rteLarge}{\texttt{rte6470}}
\newcommand{\pegaseLarge}{\texttt{pegase9k}}
\newcommand{\pegaseXLarge}{\texttt{pegase13k}}
\newcommand{\gocXXL}{\texttt{goc30k}}

This section presents the numerical experiments used to assess
E2ELR. The experiments are conducted on power grids with up to 30,000
buses and uses two variants of Problem \ref{eq:DCOPF} with and without
reserve requirements.  The section presents the data-generation
methodology, the baseline ML architectures, performance metrics, and
implementation details.  Additional information is in \cite{Chen2023_E2ELR_arxiv}.

\subsection{Data Generation}
\label{sec:experiment:data}

    \begin{table}[!t]
        \centering
        \caption{Selected test cases from PGLib \cite{PGLib}}
        \label{tab:PGLib}
        \begin{tabular}{lcrrrrr}
            \toprule
            System & \revision{}{Ref}
                & \multicolumn{1}{c}{$|\mathcal{N}|$}
                & \multicolumn{1}{c}{$|\mathcal{E}|$}
                & \multicolumn{1}{c}{$|\mathcal{G}|$} 
                & \multicolumn{1}{c}{${D_{\text{ref}}}^{\dagger}$}
                & \multicolumn{1}{c}{$\alpha_{\text{r}}$}
                \\
            \midrule
            \ieeeSmall      & \cite{UW_PowerSystemArchive} & 300   &	411 &	69  & 23.53    & 34.16\% \\
            \pegaseSmall    & \cite{Josz2016_ACOPF_PegaseRTE} & 1354  & 1991  & 260   & 73.06    & 19.82\% \\
            \rteLarge       & \cite{Josz2016_ACOPF_PegaseRTE} & 6470  & 9005  & 761   & 96.59    & 14.25\% \\
            \pegaseLarge    & \cite{Josz2016_ACOPF_PegaseRTE} & 9241  & 16049 & 1445  & 312.35   &  4.70\% \\
            \pegaseXLarge   & \cite{Josz2016_ACOPF_PegaseRTE} & 13659 & 20467 & 4092  & 381.43   &  1.32\% \\
            \gocXXL         & \cite{GoCompetition} & 30000 & 35393 & 3526  & 117.74   &  4.68\% \\
            \bottomrule
        \end{tabular}\\
        \footnotesize{$^{\dagger}$Total active power demand in reference PGLib case, in GW.}
    \end{table}
    
Instances of Problem \eqref{eq:DCOPF} are obtained by perturbing
reference test cases from the PGLib \cite{PGLib} library \revision{}{(v21.07)}.  Two
categories of instances are generated: instances without any reserve
requirements (ED), and instances with reserve requirements (ED-R).
The instances are generated as follows. Denote by $\pd^{\text{ref}}$
the nodal load vector from the reference PGLib case.  ED instances
are obtained by perturbing this reference load profile.
Namely, for instance $i$, $\pd^{(i)} \, {=} \, \gamma^{(i)} \times \eta^{(i)} \times \pd^{\text{ref}}$,
where $\gamma^{(i)} \, {\in} \, \mathbb{R}$ is a global scaling
factor, $\eta \, {\in} \, \mathbb{R}^{|\mathcal{N}|}$ denotes
load-level multiplicative white noise, and the multiplications are
element-wise.  For the case at hand, $\gamma$ is sampled from a
uniform distribution $U[0.8, 1.2]$,
and, for each load, $\eta$ is sampled from a log-normal
distribution with mean $1$ and standard deviation $5\%$.

ED-R instances are identical to the ED instances, except that
reserve requirements are set to a non-zero value.  The PGLib library
does not include reserve information, therefore, the paper assumes
$\bar{r}_{g} = \alpha_{\text{r}} \bar{p}_{g}, \forall g \in
\mathcal{G}$, where
$
        \alpha_{\text{r}} = 5 \times \| \bar{\pg} \|_{\infty} \times \| \bar{\pg} \|_{1}^{-1}.
$        
This ensures that the total reserve capacity is 5 times larger than
the largest generator in the system.  Then, the reserve requirements
of each instance is sampled uniformly between 100\% and 200\% of the
size of the largest generator, thereby mimicking contingency reserve
requirements used in industry.

Table \ref{tab:PGLib} presents the systems used in the experiments.
The table reports: the number of buses ($|\mathcal{N}|$), the number
of branches ($|\mathcal{E}|$), the number of generators
($|\mathcal{G}|$), the total active power demand in the reference
PGLib case ($D_{\text{ref}}$, in GW), and the value of
$\alpha_{\text{r}}$ used to determine reserve capacities.  The
experiments consider test cases with up to 30,000 buses, significantly
larger than almost all previous works.  Large systems have a smaller
value of $\alpha_{\text{r}}$ because they contain significantly more
generators, whereas the size of the largest generator typically
remains in the same order of magnitude.  For every test case, $50,000$
instances are generated and solved using Gurobi.  This dataset is then
split into training, validation, and test sets which comprise $40000$,
$5000$, and $5000$ instances.

\subsection{Baseline Models}
\label{sec:experiment:baselines}

The proposed end-to-end learning and repair model (E2ELR) is evaluated
against \revision{three}{four} architectures.  First, a naive, fully-connected DNN
model without any feasibility layer (DNN).  This model only includes a
sigmoid activation layer to enforce generation bounds (constraint
\eqref{eq:DCOPF:dispatch_bounds}). 
\revision{blue}{Second, a fully-connected DNN model with the DeepOPF architecture \cite{pan2020deepopf} (DeepOPF).
It uses an equality completion to ensure the satisfaction of equality constraints; the output may violate inequality constraints.}
\revision{Second}{Third}, a fully-connected DNN
model with the DC3 architecture \cite{donti2021dc3} (DC3).  This
architecture uses a fixed-step unrolled gradient descent to minimize
constraint violations; it is however not guaranteed to reach zero
violations.  Note that the DC3 architecture requires a significant
amount of hypertuning to achieve decent results.  
\revision{Third,}{The last model is} a
fully-connected DNN model, combined with the LOOP-LC architecture from
\cite{Li2022_LOOP-LC} (LOOP).  The gauge mapping used in LOOP does not
support the compact form of Eq. \eqref{eq:reserves:sufficient}, therefore it is not included in the ED-R experiments.  These baseline
models are detailed in \cite{Chen2023_E2ELR_arxiv}.

\revision{}{All baselines use a fully-connected DNN architecture, with the main difference being how feasibility is handled.
Graph Neural Network (GNN) architectures are not considered in this work, as they were found to be numerically less stable and exhibited poorer performance than DNNs in preliminary experiments.
Nevertheless, note that the proposed repair layers can also be used in conjunction with a GNN architecture.}

\subsection{Performance Metrics}
\label{sec:experiment:metrics}

The performance of each ML model is evaluated with respect to several
metrics that measure both accuracy and computational efficiency.
Given an instance $\x$ with optimal solution $\pg^{*}$ and a predicted
solution $\hat{\pg}$, the optimality gap is defined as
$
        \text{gap} = (\hat{Z} - Z^{*}) \times | Z^{*} |^{-1},
$
    where $Z^{*}$ is the optimal value of the problem, and $\hat{Z}$ is the objective value of the prediction, plus a penalty for hard constraint violations, i.e., 
    \begin{align}
        \label{eq:experiment:objective_penalized}
        c(\hat{\pg}) + \Mth \| \xith(\hat{\pg}) \|_{1} + \Mpb |\mathbf{e}^{\top}(\hat{\pg} - \pd)| + \Mres \xir(\hat{\pg}),
    \end{align}
where $\xir(\hat{\pg})$ is defined as in function
\eqref{eq:SL:reserve_shortage}.  Penalizing hard constraint violations
is necessary to ensure a fair comparison between models that output
feasible solutions and those that do not.  Because all considered
models enforce constraints
\eqref{eq:DCOPF:eco_max}--\eqref{eq:DCOPF:reserve_bounds}, they are
not penalized in Eq. \eqref{eq:experiment:objective_penalized}.

The paper uses realistic penalty prices, based on the values used by
MISO in their operations
\cite{MISO_Schedule28_ReserveDemandCurve,MISO_Schedule28A_TranmissionDemandCurve}.
Namely, the thermal violation penalty price $\Mth$ is set to
\$1500/MW.  The power balance violation penalty $\Mpb$ is set to
\$3500/MW, which corresponds to MISO's value of lost load (VOLL).
Finally, the reserve shortage penalty $\Mres$ is set to \$1100/MW,
which is MISO's reserve shortage price. The ability of optimization
proxies to output feasible solution is measured via the proportion of
feasible predictions, which is reported as a percentage over the test
set.  The paper uses an absolute tolerance of $10^{-4}$ p.u. to decide
whether a constraint is violated; note that this is 100x larger than
the default absolute tolerance of optimization solvers.  The paper
also reports the mean constraint violation of infeasible predictions.

The paper also evaluates the computational efficiency of each ML
model and learning paradigm (SL and SSL), as well as that of the
repair layers.  Computational efficiency is measured by (i) the
training time of ML models, including the data-generation time when
applicable, and (ii) the inference time.  Note that ML models evaluate
\emph{batches} of instances, therefore, inference times are reported
per batch of 256 instances.  The performance of the repair layers
presented in Section \ref{sec:layers} is benchmarked against a
standard euclidean projection solved with state-of-the-art
optimization software.

Unless specified otherwise, average computing times are arithmetic
means; other averages are shifted geometric means
\begin{align*}
    \mu_{s}(x_{1}, ..., x_{n}) = \exp \Big( \frac{1}{n} \sum_{i} \log (x_{i} + s) \Big) - s.
\end{align*}
The paper uses a shift $s$ of 1\% for optimality gaps, and 1p.u. for
constraint violations.

\subsection{Implementation Details}
\label{sec:experiment:implementation}

All optimization problems are formulated in Julia using JuMP
\cite{Dunning2017_JuMP}, and solved with Gurobi 9.5 \cite{gurobi} with
a single CPU thread and default parameter settings.  All deep learning
models are implemented using PyTorch \cite{paszke2017automatic} and
trained using the Adam optimizer \cite{kingma2014adam}.  All models
are hyperparameter tuned using a grid search, which is detailed in
\cite{Chen2023_E2ELR_arxiv}.  For each system, the
best model is selected on the validation set and the performances on
the test set are reported.
\revision{}{During training, the learning rate is reduced by a factor of ten if the validation loss shows no improvement for a consecutive sequence of 10 epochs.
In addition, training is stopped if the validation loss does not improve for consecutive 20 epochs.}
Experiments are conducted on dual Intel
Xeon 6226@2.7GHz machines running Linux, on the PACE Phoenix cluster
\cite{PACE}.  The training of ML models is performed on Tesla
V100-PCIE GPUs with 16GBs HBM2 RAM.

\section{Numerical Results}
\label{sec:results}

\subsection{Optimality Gaps}
\label{sec:results:gaps}

Table \ref{tab:res:opt_gap} reports, for the ED and ED-R problems, the
mean optimality gap of each ML model, under the SL and SSL learning
modes. Bold entries denote the best-performing method. Recall that
LOOP is not included in ED-R experiments. {\em E2ELR systematically
  outperforms all other baselines across all settings.}  This stems
from two reasons. First, DNN, \revision{}{DeepOPF} and DC3 exhibit violations of the power
balance constraint \eqref{eq:DCOPF:power_balance}, which yields high
penalties and therefore large optimality gaps.  Statistics on power
balance violations for DNN, \revision{}{DeepOPF} and DC3 are reported in Table
\ref{tab:res:violation}. Second, LOOP's poor performance, despite
not violating any hard constraint, is because the non-convex gauge mapping
used inside the model has an adverse impact on training.
Indeed, after a few epochs of training, LOOP gets stuck
in a local optimum.

    \begin{table}[!t]
        \centering
        \caption{Mean optimality Gap (\%) on test set}
        \label{tab:res:opt_gap}
        \resizebox{\columnwidth}{!}{
        \begin{tabular}{llrrrrrrrrr}
\toprule
&
    & \multicolumn{5}{c}{ED}
    & \multicolumn{4}{c}{ED-R}\\
\cmidrule(lr){3-7}
\cmidrule(lr){8-11}
Loss & System         
        &      DNN &    E2ELR &      \revision{}{DeepOPF} & DC3 &     LOOP
        &      DNN &    E2ELR &      \revision{}{DeepOPF} & DC3 \\
\midrule
SL   & \ieeeSmall     
        &    69.55 &     \textbf{1.42} &           \revision{}{2.81} &     3.03 &    38.93
        &    75.06 &     \textbf{1.52} &           \revision{}{2.80} &     2.94\\
     & \pegaseSmall   
        &    48.77 &     \textbf{0.74} &           \revision{}{7.78} &     2.80 &    32.53
        &    47.84 &     \textbf{0.74} &           \revision{}{7.50} &     2.97\\
     & \rteLarge           
        &    55.13 &     \textbf{1.35} &          \revision{}{28.23} &     3.68 &    50.21
        &    70.57 &     \textbf{1.82} &          \revision{}{46.66} &     3.49\\
     & \pegaseLarge        
        &    76.06 &     \textbf{0.38} &          \revision{}{33.20} &     1.25 &    33.78
        &    81.19 &     \textbf{0.38} &          \revision{}{30.84} &     1.29\\
     & \pegaseXLarge       
        &    71.14 &     \textbf{0.29} &          \revision{}{64.93} &     1.79 &    52.94
        &    76.32 &     \textbf{0.28} &          \revision{}{69.23} &     1.81\\
     & \gocXXL             
        &   194.13 &     \textbf{0.46} &          \revision{}{55.91} &     2.75 &    36.49
        &   136.25 &     \textbf{0.45} &          \revision{}{41.34} &     2.35\\
\midrule
SSL  & \ieeeSmall          
        &    35.66 &     \textbf{0.74} &          \revision{}{ 2.23} &     2.51 &    37.78
        &    45.56 &     \textbf{0.78} &          \revision{}{ 2.82} &     2.80\\
     & \pegaseSmall        
        &    62.07 &     \textbf{0.63} &          \revision{}{10.83} &     2.57 &    32.20
        &    64.69 &     \textbf{0.68} &          \revision{}{ 9.83} &     2.61\\
     & \rteLarge           
        &    40.73 &     \textbf{1.30} &          \revision{}{42.28} &     2.82 &    50.20
        &    55.16 &     \textbf{1.68} &          \revision{}{48.57} &     3.04\\
     & \pegaseLarge        
        &    43.68 &     \textbf{0.32} &          \revision{}{34.33} &     0.82 &    33.76
        &    44.74 &     \textbf{0.29} &          \revision{}{42.06} &     0.93\\
     & \pegaseXLarge       
        &    57.58 &     \textbf{0.21} &          \revision{}{60.12} &     0.84 &    52.93
        &    61.28 &     \textbf{0.19} &          \revision{}{65.38} &     0.91\\
     & \gocXXL             
        &   108.91 &     \textbf{0.39} &          \revision{}{ 8.39} &     0.72 &    36.73
        &    93.91 &     \textbf{0.33} &          \revision{}{ 9.47} &     0.71\\     
\bottomrule
\end{tabular}

        }\\
    \end{table}

        \begin{table}[!t]
        \centering
        \caption{Power balance constraint violation statistics}
        \label{tab:res:violation}
        \resizebox{\columnwidth}{!}{%
\begin{tabular}{llrrrrrrrrrrrr}
\toprule
&
    & \multicolumn{6}{c}{ED}
    & \multicolumn{6}{c}{ED-R}\\
\cmidrule(lr){3-8}
\cmidrule(lr){9-14}
     &               
     & \multicolumn{2}{c}{DNN}
     & \multicolumn{2}{c}{\revision{}{DeepOPF}}
     & \multicolumn{2}{c}{DC3}
     & \multicolumn{2}{c}{DNN}
     & \multicolumn{2}{c}{\revision{}{DeepOPF}}
     & \multicolumn{2}{c}{DC3}\\
\cmidrule(lr){3-4}
\cmidrule(lr){5-6}
\cmidrule(lr){7-8}
\cmidrule(lr){9-10}
\cmidrule(lr){11-12}
\cmidrule(lr){13-14}
Loss & System        
    & \multicolumn{1}{c}{\%feas} & \multicolumn{1}{c}{viol}
    & \multicolumn{1}{c}{\%feas} & \multicolumn{1}{c}{viol}
    & \multicolumn{1}{c}{\%feas} & \multicolumn{1}{c}{viol}
    & \multicolumn{1}{c}{\%feas} & \multicolumn{1}{c}{viol}
    & \multicolumn{1}{c}{\%feas} & \multicolumn{1}{c}{viol}
    & \multicolumn{1}{c}{\%feas} & \multicolumn{1}{c}{viol}\\
\midrule
SL   & \ieeeSmall    
         &      0\% &   0.50  &    \revision{}{100\%} &   \revision{}{0.00}  &    100\% &   0.00 
         &      0\% &   0.70  &    \revision{}{100\%} &   \revision{}{0.06}  &    100\% &   0.25 \\
     & \pegaseSmall  
         &      0\% &   1.82  &     \revision{}{71\%} &   \revision{}{1.35}  &     65\% &   0.04 
         &      0\% &   1.90  &     \revision{}{70\%} &   \revision{}{1.19}  &     65\% &   0.08 \\
     & \rteLarge     
         &      0\% &   2.08  &      \revision{}{1\%} &   \revision{}{1.94}  &      9\% &   0.03 
         &      0\% &   2.26  &      \revision{}{1\%} &   \revision{}{2.81}  &      7\% &   0.01 \\
     & \pegaseLarge  
         &      0\% &   6.25  &      \revision{}{3\%} &   \revision{}{6.38}  &     29\% &   0.00 
         &      0\% &   5.91  &      \revision{}{3\%} &   \revision{}{5.99}  &     29\% &   0.00 \\
     & \pegaseXLarge 
         &      0\% &   6.63  &      \revision{}{1\%} &   \revision{}{7.08}  &     23\% &   0.00 
         &      0\% &   7.79  &      \revision{}{1\%} &   \revision{}{7.46}  &     22\% &   0.00 \\
     & \gocXXL       
         &      0\% &   3.81  &     \revision{}{45\%} &   \revision{}{1.70}  &     57\% &   0.03 
         &      0\% &   2.31  &     \revision{}{54\%} &   \revision{}{1.54}  &     69\% &   0.00 \\
\midrule
SSL  & \ieeeSmall    
         &      0\% &   0.58  &    \revision{}{100\%} &   \revision{}{0.00}  &    100\% &   0.16 
         &      0\% &   0.73  &    \revision{}{100\%} &   \revision{}{0.56}  &    100\% &   0.29 \\
     & \pegaseSmall  
         &      0\% &   2.42  &     \revision{}{63\%} &   \revision{}{2.56}  &     63\% &   0.03 
         &      0\% &   2.30  &     \revision{}{63\%} &   \revision{}{1.81}  &     39\% &   0.05 \\
     & \rteLarge     
         &      0\% &   1.90  &      \revision{}{1\%} &   \revision{}{2.64}  &      6\% &   0.11 
         &      0\% &   2.72  &      \revision{}{1\%} &   \revision{}{2.51}  &      5\% &   0.01 \\
     & \pegaseLarge  
         &      0\% &   7.22  &      \revision{}{3\%} &   \revision{}{5.46}  &     27\% &   0.05 
         &      0\% &   7.00  &      \revision{}{3\%} &   \revision{}{7.87}  &     25\% &   0.08 \\
     & \pegaseXLarge 
         &      0\% &   6.41  &      \revision{}{1\%} &   \revision{}{7.19}  &     19\% &   0.01 
         &      0\% &   6.82  &      \revision{}{1\%} &   \revision{}{7.83}  &     20\% &   0.01 \\
     & \gocXXL       
         &      0\% &   3.18  &     \revision{}{55\%} &   \revision{}{1.24}  &     52\% &   0.01 
         &      0\% &   2.61  &     \revision{}{49\%} &   \revision{}{1.20}  &     62\% &   0.02 \\
\bottomrule
\end{tabular}

        }
        \footnotesize{$^{\dagger}$with 200 gradient steps. $^{*}$geometric mean of non-zero violations, in p.u.}
        \end{table}

{\em E2ELR, when trained in a self-supervised mode, achieves the best
  performance.}  This is because SSL directly minimizes the true
objective function, rather than the surrogate supervised loss.  With
the exception of \rteLarge, the performance of E2ELR improves as the
size of the system increases, with the lowest optimality gaps achieved
on \pegaseXLarge, which has the most generators.  Note that \rteLarge
is a real system from the French transmission grid: it is more
congested than other test cases, and therefore harder to learn.

\subsection{Computing Times}
\label{sec:results:times}

Tables \ref{tab:exp:ED:train_time} and \ref{tab:exp:ED-R:train_time}
report the sampling and training times for ED and ED-R, respectively.
Each table reports the total time for data-generation, which
corresponds to the total solving time of Gurobi on a single thread.
There is no labeling time for self-supervised models.  While training
times for SL and SSL are comparable, for a given architecture, the
latter does not incur any labeling time.  The training time of DC3 is
significantly higher than other baselines because of its unrolled
gradient steps.  These results demonstrate that ML models can be
trained efficiently on large-scale systems.  Indeed, {\em the
  self-supervised E2ELR needs less than an hour of total computing
  time to achieve optimality gaps under $0.5\%$ for systems with
  thousands of buses.}

    \begin{table}[!t]
        \centering
        \caption{Sampling and training time comparison (ED)}
        \label{tab:exp:ED:train_time}
        \resizebox{\columnwidth}{!}{
        \begin{tabular}{llrrrrrrrrrrr}
            \toprule
            Loss & System & \multicolumn{1}{c}{Sample} & \multicolumn{1}{c}{DNN} & \multicolumn{1}{c}{E2ELR} & \multicolumn{1}{c}{\revision{}{DeepOPF}} &\multicolumn{1}{c}{DC3} & \multicolumn{1}{c}{LOOP}  \\
            \midrule
            SL
            & \ieeeSmall       & \qty[mode=text]{ 0.2}{\hour} & \qty[mode=text]{ 7}{\minute}  & \qty[mode=text]{37}{\minute} & \revision{}{\qty[mode=text]{31}{\minute}} & \qty[mode=text]{121}{\minute} & \qty[mode=text]{ 33}{\minute} \\\
            & \pegaseSmall     & \qty[mode=text]{ 0.7}{\hour} & \qty[mode=text]{ 8}{\minute}  & \qty[mode=text]{14}{\minute} & \revision{}{\qty[mode=text]{ 6}{\minute}} & \qty[mode=text]{ 41}{\minute} & \qty[mode=text]{ 19}{\minute} \\\
            & \rteLarge        & \qty[mode=text]{ 5.1}{\hour} & \qty[mode=text]{11}{\minute}  & \qty[mode=text]{30}{\minute} & \revision{}{\qty[mode=text]{13}{\minute}} & \qty[mode=text]{ 73}{\minute} & \qty[mode=text]{ 18}{\minute}\\
            & \pegaseLarge     & \qty[mode=text]{12.7}{\hour} & \qty[mode=text]{15}{\minute}  & \qty[mode=text]{24}{\minute} & \revision{}{\qty[mode=text]{22}{\minute}} & \qty[mode=text]{123}{\minute} & \qty[mode=text]{ 25}{\minute}\\
            & \pegaseXLarge    & \qty[mode=text]{20.6}{\hour} & \qty[mode=text]{14}{\minute}  & \qty[mode=text]{19}{\minute} & \revision{}{\qty[mode=text]{14}{\minute}} & \qty[mode=text]{126}{\minute} & \qty[mode=text]{ 19}{\minute}\\
            & \gocXXL          & \qty[mode=text]{63.4}{\hour} & \qty[mode=text]{25}{\minute}  & \qty[mode=text]{20}{\minute} & \revision{}{\qty[mode=text]{41}{\minute}} & \qty[mode=text]{108}{\minute} & \qty[mode=text]{127}{\minute}\\
            \midrule
            SSL
            & \ieeeSmall       & -- &  \qty[mode=text]{15}{\minute} & \qty[mode=text]{27}{\minute} & \revision{}{\qty[mode=text]{38}{\minute}} & \qty[mode=text]{102}{\minute} & \qty[mode=text]{27}{\minute} \\
            & \pegaseSmall     & -- &  \qty[mode=text]{ 8}{\minute} & \qty[mode=text]{15}{\minute} & \revision{}{\qty[mode=text]{11}{\minute}} & \qty[mode=text]{ 46}{\minute} & \qty[mode=text]{14}{\minute} \\
            & \rteLarge        & -- &  \qty[mode=text]{ 9}{\minute} & \qty[mode=text]{17}{\minute} & \revision{}{\qty[mode=text]{10}{\minute}} & \qty[mode=text]{ 42}{\minute} & \qty[mode=text]{15}{\minute} \\
            & \pegaseLarge     & -- &  \qty[mode=text]{18}{\minute} & \qty[mode=text]{20}{\minute} & \revision{}{\qty[mode=text]{17}{\minute}} & \qty[mode=text]{100}{\minute} & \qty[mode=text]{29}{\minute} \\
            & \pegaseXLarge    & -- &  \qty[mode=text]{17}{\minute} & \qty[mode=text]{18}{\minute} & \revision{}{\qty[mode=text]{14}{\minute}} & \qty[mode=text]{125}{\minute} & \qty[mode=text]{15}{\minute} \\
            & \gocXXL          & -- &  \qty[mode=text]{38}{\minute} & \qty[mode=text]{45}{\minute} & \revision{}{\qty[mode=text]{51}{\minute}} & \qty[mode=text]{105}{\minute} & \qty[mode=text]{60}{\minute} \\
            \bottomrule
        \end{tabular}
        }
        \footnotesize{Sampling (training) times are for 1 CPU (1 GPU). Excludes hypertuning.}
    \end{table}

    \begin{table}[!t]
        \centering
        \caption{Sampling and training time comparison (ED-R)}
        \label{tab:exp:ED-R:train_time}
        \resizebox{\columnwidth}{!}{
        \begin{tabular}{llrrrrr}
            \toprule
            Loss & System & \multicolumn{1}{c}{Sample} & \multicolumn{1}{c}{DNN} & \multicolumn{1}{c}{E2ELR} & \multicolumn{1}{c}{\revision{}{DeepOPF}} & \multicolumn{1}{c}{DC3} \\
            \midrule
            SL
            & \ieeeSmall       & \qty[mode=text]{ 0.2}{\hour} & \qty[mode=text]{12}{\minute}  & \qty[mode=text]{43}{\minute} & \revision{}{\qty[mode=text]{43}{\minute}} & \qty[mode=text]{115}{\minute} \\
            & \pegaseSmall     & \qty[mode=text]{ 0.8}{\hour} & \qty[mode=text]{14}{\minute}  & \qty[mode=text]{19}{\minute} & \revision{}{\qty[mode=text]{19}{\minute}} & \qty[mode=text]{53 }{\minute} \\
            & \rteLarge        & \qty[mode=text]{ 4.6}{\hour} & \qty[mode=text]{14}{\minute}  & \qty[mode=text]{19}{\minute} & \revision{}{\qty[mode=text]{19}{\minute}} & \qty[mode=text]{71 }{\minute} \\
            & \pegaseLarge     & \qty[mode=text]{14.0}{\hour} & \qty[mode=text]{15}{\minute}  & \qty[mode=text]{22}{\minute} & \revision{}{\qty[mode=text]{22}{\minute}} & \qty[mode=text]{123}{\minute} \\
            & \pegaseXLarge    & \qty[mode=text]{22.7}{\hour} & \qty[mode=text]{16}{\minute}  & \qty[mode=text]{27}{\minute} & \revision{}{\qty[mode=text]{27}{\minute}} & \qty[mode=text]{126}{\minute} \\
            & \gocXXL          & \qty[mode=text]{65.9}{\hour} & \qty[mode=text]{32}{\minute}  & \qty[mode=text]{39}{\minute} & \revision{}{\qty[mode=text]{38}{\minute}} & \qty[mode=text]{129}{\minute} \\
            \midrule
            SSL
            & \ieeeSmall       & -- & \qty[mode=text]{21}{\minute}  & \qty[mode=text]{37}{\minute} & \revision{}{\qty[mode=text]{37}{\minute}} & \qty[mode=text]{131}{\minute} \\
            & \pegaseSmall     & -- & \qty[mode=text]{ 6}{\minute}  & \qty[mode=text]{19}{\minute} & \revision{}{\qty[mode=text]{19}{\minute}} & \qty[mode=text]{ 67}{\minute} \\
            & \rteLarge        & -- & \qty[mode=text]{12}{\minute}  & \qty[mode=text]{21}{\minute} & \revision{}{\qty[mode=text]{21}{\minute}} & \qty[mode=text]{ 71}{\minute} \\
            & \pegaseLarge     & -- & \qty[mode=text]{20}{\minute}  & \qty[mode=text]{24}{\minute} & \revision{}{\qty[mode=text]{24}{\minute}} & \qty[mode=text]{123}{\minute} \\
            & \pegaseXLarge    & -- & \qty[mode=text]{13}{\minute}  & \qty[mode=text]{22}{\minute} & \revision{}{\qty[mode=text]{22}{\minute}} & \qty[mode=text]{125}{\minute} \\
            & \gocXXL          & -- & \qty[mode=text]{52}{\minute}  & \qty[mode=text]{53}{\minute} & \revision{}{\qty[mode=text]{53}{\minute}} & \qty[mode=text]{128}{\minute} \\
            \bottomrule
        \end{tabular}
        }\\
        \footnotesize{Sampling (training) times are for 1 CPU (1 GPU). Excludes hypertuning.}
    \end{table}

Tables \ref{tab:exp:ED:inference_time} and
\ref{tab:exp:ED-R:inference_time} report, for ED and ED-R,
respectively, the average solving time using Gurobi (GRB) and average
inference times of ML methods.  Recall that the Gurobi's solving times
are for a single instance solved on a single CPU core, whereas the ML
inference times are reported for a batch of 256 instances on a GPU.
Also note that the number of gradient steps used by DC3 to recover
feasibility is set to 200 for inference (compared to 50 for training).

On systems with more than 6,000 buses, DC3 is typically 10--30 times
slower than other baselines, again due to its unrolled gradient steps.
\revision{blue}{In contrast, DNN, DeepOPF, E2ELR, and LOOP all require in the
order of 5--10 milliseconds to evaluate a batch of 256 instances.}  For
the largest systems, this represents about 25,000 instances per second,
on a single GPU.  Solving the same volume of instances with Gurobi would require more than
a day on a single CPU.
Getting this time down to the order of seconds, thereby matching the speed of ML proxies, would require thousands of
CPUs, which comes at high financial and environmental costs.

    \begin{table}[!t]
        \centering
        \caption{Solving and inference time comparison (ED)}
        \label{tab:exp:ED:inference_time}
        \resizebox{\columnwidth}{!}{
            \begin{tabular}{llrrrrrrr}
\toprule
Loss & System         &      DNN &    E2ELR &   \revision{}{DeepOPF} &  DC3$^{\dagger}$ &   LOOP & GRB$^{*}$\\
\midrule
SL 
    & \ieeeSmall
        & \qty[mode=text]{3.4}{\ms}
        & \qty[mode=text]{4.5}{\ms}
        & \revision{}{\qty[mode=text]{4.8}{\ms}}
        & \qty[mode=text]{15.4}{\ms}
        & \qty[mode=text]{5.3}{\ms}
        & \qty[mode=text]{12.1}{\ms} \\
    & \pegaseSmall 
        & \qty[mode=text]{4.1}{\ms}
        & \qty[mode=text]{5.3}{\ms} 
        & \revision{}{\qty[mode=text]{4.3}{\ms}}
        & \qty[mode=text]{18.3}{\ms} 
        & \qty[mode=text]{5.9}{\ms} 
        & \qty[mode=text]{51.5}{\ms} \\
    & \rteLarge 
        & \qty[mode=text]{5.1}{\ms} 
        & \qty[mode=text]{6.6}{\ms} 
        & \revision{}{\qty[mode=text]{5.4}{\ms}}
        & \qty[mode=text]{35.3}{\ms} 
        & \qty[mode=text]{7.1}{\ms} 
        & \qty[mode=text]{364.4}{\ms} \\
    & \pegaseLarge 
        & \qty[mode=text]{6.0}{\ms} 
        & \qty[mode=text]{7.3}{\ms} 
        & \revision{}{\qty[mode=text]{6.2}{\ms}}
        & \qty[mode=text]{91.5}{\ms} 
        & \qty[mode=text]{8.2}{\ms} 
        & \qty[mode=text]{913.5}{\ms} \\
    & \pegaseXLarge 
        & \qty[mode=text]{7.3}{\ms} 
        & \qty[mode=text]{8.3}{\ms} 
        & \revision{}{\qty[mode=text]{8.7}{\ms}}
        & \qty[mode=text]{523.6}{\ms} 
        & \qty[mode=text]{13.9}{\ms} 
        & \qty[mode=text]{1481.3}{\ms} \\
    & \gocXXL 
        & \qty[mode=text]{9.5}{\ms} 
        & \qty[mode=text]{10.0}{\ms} 
        & \revision{}{\qty[mode=text]{9.3}{\ms}}
        & \qty[mode=text]{443.0}{\ms} 
        & \qty[mode=text]{14.4}{\ms} 
        & \qty[mode=text]{4566.9}{\ms} \\
\midrule
SSL 
    & \ieeeSmall 
        & \qty[mode=text]{3.4}{\ms} 
        & \qty[mode=text]{6.0}{\ms} 
        & \revision{}{\qty[mode=text]{3.7}{\ms}}
        & \qty[mode=text]{15.1}{\ms} 
        & \qty[mode=text]{5.2}{\ms} 
        & \qty[mode=text]{12.1}{\ms} \\
    & \pegaseSmall 
        & \qty[mode=text]{4.0}{\ms} 
        & \qty[mode=text]{5.3}{\ms} 
        & \revision{}{\qty[mode=text]{4.3}{\ms}}
        & \qty[mode=text]{18.4}{\ms} 
        & \qty[mode=text]{5.8}{\ms} 
        & \qty[mode=text]{51.5}{\ms} \\
    & \rteLarge 
        & \qty[mode=text]{5.9}{\ms} 
        & \qty[mode=text]{6.5}{\ms} 
        & \revision{}{\qty[mode=text]{6.3}{\ms}}
        & \qty[mode=text]{36.7}{\ms} 
        & \qty[mode=text]{9.5}{\ms} 
        & \qty[mode=text]{364.4}{\ms} \\
    & \pegaseLarge 
        & \qty[mode=text]{6.1}{\ms} 
        & \qty[mode=text]{7.0}{\ms} 
        & \revision{}{\qty[mode=text]{6.3}{\ms}}
        & \qty[mode=text]{93.2}{\ms} 
        & \qty[mode=text]{10.3}{\ms} 
        & \qty[mode=text]{913.5}{\ms} \\
    & \pegaseXLarge 
        & \qty[mode=text]{7.1}{\ms} 
        & \qty[mode=text]{8.2}{\ms} 
        & \revision{}{\qty[mode=text]{7.2}{\ms}}
        & \qty[mode=text]{561.2}{\ms} 
        & \qty[mode=text]{12.8}{\ms} 
        & \qty[mode=text]{1481.3}{\ms} \\
    & \gocXXL 
        & \qty[mode=text]{10.9}{\ms} 
        & \qty[mode=text]{11.7}{\ms} 
        & \revision{}{\qty[mode=text]{9.4}{\ms}}
        & \qty[mode=text]{444.2}{\ms} 
        & \qty[mode=text]{21.7}{\ms} 
        & \qty[mode=text]{4566.9}{\ms} \\
\bottomrule
\end{tabular}

        }\\
        \footnotesize{$^{\dagger}$with 200 gradient steps. $^{*}$solution time per instance (single thread).\\All ML inference times are for a batch of 256 instances.}
    \end{table}

    \begin{table}[!t]
        \centering
        \caption{Solving and inference time comparison (ED-R)}
        \label{tab:exp:ED-R:inference_time}
        \resizebox{\columnwidth}{!}{%
            \begin{tabular}{llrrrrr}
\toprule
Loss & System         &      DNN &    E2ELR  &   \revision{}{DeepOPF} &     DC3$^{\dagger}$ & GRB$^{*}$\\
\midrule
SL 
    & \ieeeSmall 
        & \qty[mode=text]{   3.9}{\ms} 
        & \qty[mode=text]{   6.5}{\ms}
        & \revision{}{\qty[mode=text]{4.6}{\ms}}
        & \qty[mode=text]{  16.5}{\ms} 
        & \qty[mode=text]{  12.6}{\ms} \\
    & \pegaseSmall 
        & \qty[mode=text]{   4.5}{\ms} 
        & \qty[mode=text]{   6.0}{\ms} 
        & \revision{}{\qty[mode=text]{4.8}{\ms}}
        & \qty[mode=text]{  18.9}{\ms} 
        & \qty[mode=text]{  56.5}{\ms} \\
    & \rteLarge 
        & \qty[mode=text]{   5.7}{\ms} 
        & \qty[mode=text]{  10.4}{\ms} 
        & \revision{}{\qty[mode=text]{6.2}{\ms}}
        & \qty[mode=text]{  36.1}{\ms} 
        & \qty[mode=text]{ 333.6}{\ms} \\
    & \pegaseLarge 
        & \qty[mode=text]{   6.3}{\ms} 
        & \qty[mode=text]{   7.7}{\ms} 
        & \revision{}{\qty[mode=text]{6.7}{\ms}}
        & \qty[mode=text]{  91.6}{\ms} 
        & \qty[mode=text]{1008.0}{\ms} \\
    & \pegaseXLarge 
        & \qty[mode=text]{   8.3}{\ms} 
        & \qty[mode=text]{  10.7}{\ms} 
        & \revision{}{\qty[mode=text]{8.8}{\ms}}
        & \qty[mode=text]{ 531.2}{\ms} 
        & \qty[mode=text]{1632.7}{\ms} \\
    & \gocXXL 
        & \qty[mode=text]{   9.3}{\ms} 
        & \qty[mode=text]{  11.1}{\ms}
        & \revision{}{\qty[mode=text]{10.6}{\ms}}
        & \qty[mode=text]{ 438.7}{\ms} 
        & \qty[mode=text]{4745.7}{\ms} \\
\midrule
SSL 
    & \ieeeSmall 
        & \qty[mode=text]{   3.9}{\ms}
        & \qty[mode=text]{   7.6}{\ms}
        & \revision{}{\qty[mode=text]{4.4}{\ms}}
        & \qty[mode=text]{  17.6}{\ms}
        & \qty[mode=text]{  12.6}{\ms} \\
    & \pegaseSmall 
        & \qty[mode=text]{   4.4}{\ms} 
        & \qty[mode=text]{   5.9}{\ms} 
        & \revision{}{\qty[mode=text]{4.7}{\ms}}
        & \qty[mode=text]{  19.1}{\ms} 
        & \qty[mode=text]{  56.5}{\ms} \\
    & \rteLarge 
        & \qty[mode=text]{   6.4}{\ms} 
        & \qty[mode=text]{  10.5}{\ms} 
        & \revision{}{\qty[mode=text]{6.7}{\ms}}
        & \qty[mode=text]{  37.3}{\ms} 
        & \qty[mode=text]{ 333.6}{\ms} \\
    & \pegaseLarge 
        & \qty[mode=text]{   7.1}{\ms} 
        & \qty[mode=text]{   8.3}{\ms} 
        & \revision{}{\qty[mode=text]{7.2}{\ms}}
        & \qty[mode=text]{  92.9}{\ms} 
        & \qty[mode=text]{1008.0}{\ms} \\
    & \pegaseXLarge 
        & \qty[mode=text]{   7.8}{\ms} 
        & \qty[mode=text]{   8.9}{\ms} 
        & \revision{}{\qty[mode=text]{7.9}{\ms}}
        & \qty[mode=text]{ 522.4}{\ms} 
        & \qty[mode=text]{1632.7}{\ms} \\
    & \gocXXL 
        & \qty[mode=text]{  10.2}{\ms} 
        & \qty[mode=text]{  12.4}{\ms} 
        & \revision{}{\qty[mode=text]{10.3}{\ms}}
        & \qty[mode=text]{ 435.8}{\ms} 
        & \qty[mode=text]{4745.7}{\ms} \\
\bottomrule
\end{tabular}

        }
        \footnotesize{$^{\dagger}$with 200 gradient steps. $^{*}$solution time per instance (single thread).\\All ML inference times are for a batch of 256 instances.}
    \end{table}

\subsection{Benefits of End-to-End Training}
\label{sec:results:benefits_e2e}

    \begin{table}[!t]
        \centering
        \caption{Comparison of optimality gaps (\%) with and without feasibility restoration (ED)}
        \label{tab:exp:ED:opt_gap:FFR}
        \resizebox{\columnwidth}{!}{\begin{tabular}{llrrrrrrrrrr}
\toprule
&& & \multicolumn{3}{c}{DNN} & \multicolumn{3}{c}{\revision{}{DeepOPF}} & \multicolumn{3}{c}{DC3}\\
\cmidrule(lr){4-6}
\cmidrule(lr){7-9}
\cmidrule(lr){10-12}
Loss & System         &    E2ELR & \multicolumn{1}{c}{--} & \multicolumn{1}{c}{FL} & \multicolumn{1}{c}{EP} & \multicolumn{1}{c}{--} & \multicolumn{1}{c}{FL} & \multicolumn{1}{c}{EP} & \multicolumn{1}{c}{--} & \multicolumn{1}{c}{FL} & \multicolumn{1}{c}{EP}\\
\midrule
SL   & \ieeeSmall     & \textbf{    1.42} &    69.55 &    36.33 &    36.37 &     \revision{}{2.81} &     \revision{}{2.81} &     \revision{}{2.81} &     3.03 &     3.03 &     3.03\\
     & \pegaseSmall   & \textbf{    0.74} &    48.77 &     3.98 &     3.94 &     \revision{}{7.78} &     \revision{}{5.24} &     \revision{}{5.28} &     2.80 &     2.44 &     2.44\\
     & \rteLarge      & \textbf{    1.35} &    55.13 &    21.28 &    21.41 &    \revision{}{28.23} &     \revision{}{1.89} &     \revision{}{1.89} &     3.68 &     3.36 &     3.36\\
     & \pegaseLarge   & \textbf{    0.38} &    76.06 &    34.61 &    34.65 &    \revision{}{33.20} &     \revision{}{1.94} &     \revision{}{1.99} &     1.25 &     1.24 &     1.24\\
     & \pegaseXLarge  & \textbf{    0.29} &    71.14 &    32.70 &    32.72 &    \revision{}{64.93} &    \revision{}{23.36} &    \revision{}{23.36} &     1.79 &     1.79 &     1.79\\
     & \gocXXL        & \textbf{    0.46} &   194.13 &    57.53 &    57.41 &    \revision{}{55.91} &    \revision{}{30.19} &    \revision{}{30.16} &     2.75 &     2.45 &     2.45\\
\midrule
SSL  & \ieeeSmall     & \textbf{    0.74} &    35.66 &     3.82 &     3.73 &     \revision{}{2.23} &     \revision{}{2.23} &     \revision{}{2.23} &     2.51 &     2.51 &     2.51\\
     & \pegaseSmall   & \textbf{    0.63} &    62.07 &     3.24 &     3.25 &    \revision{}{10.83} &     \revision{}{3.21} &     \revision{}{3.18} &     2.57 &     2.35 &     2.35\\
     & \rteLarge      & \textbf{    1.30} &    40.73 &    11.52 &    11.47 &    \revision{}{42.28} &     \revision{}{5.38} &     \revision{}{5.37} &     2.82 &     2.10 &     2.09\\
     & \pegaseLarge   & \textbf{    0.32} &    43.68 &     3.20 &     3.22 &    \revision{}{34.33} &     \revision{}{4.73} &     \revision{}{4.74} &     0.82 &     0.64 &     0.64\\
     & \pegaseXLarge  & \textbf{    0.21} &    57.58 &    20.59 &    20.59 &    \revision{}{60.12} &    \revision{}{18.85} &    \revision{}{18.84} &     0.84 &     0.81 &     0.81\\
     & \gocXXL        & \textbf{    0.39} &   108.91 &     7.89 &     7.89 &     \revision{}{8.39} &     \revision{}{3.06} &     \revision{}{3.06} &     0.72 &     0.62 &     0.62\\
\bottomrule
\end{tabular}
}
    \end{table}

    \begin{table}[!t]
        \centering
        \caption{Comparison of optimality gaps (\%) with and without feasibility restoration (ED-R)}
        \label{tab:exp:ED-R:opt_gap:FFR}
        \resizebox{\columnwidth}{!}{\begin{tabular}{llrrrrrrrrrr}
\toprule
&& & \multicolumn{3}{c}{DNN} & \multicolumn{3}{c}{\revision{}{DeepOPF}} & \multicolumn{3}{c}{DC3}\\
\cmidrule(lr){4-6}
\cmidrule(lr){7-9}
\cmidrule(lr){10-12}
Loss & System         &    E2ELR & \multicolumn{1}{c}{--} & \multicolumn{1}{c}{FL} & \multicolumn{1}{c}{EP} & \multicolumn{1}{c}{--} & \multicolumn{1}{c}{FL} & \multicolumn{1}{c}{EP} & \multicolumn{1}{c}{--} & \multicolumn{1}{c}{FL} & \multicolumn{1}{c}{EP}\\
\midrule
SL   & \ieeeSmall     & \textbf{    1.52} &    75.06 &    30.47 &    30.49 &     \revision{}{2.80} &     \revision{}{2.80} &     \revision{}{2.80} &     2.94 &     2.94 &     2.94\\
     & \pegaseSmall   & \textbf{    0.74} &    47.84 &     2.52 &     2.50 &     \revision{}{7.50} &     \revision{}{4.79} &     \revision{}{4.79} &     2.97 &     2.34 &     2.34\\
     & \rteLarge      & \textbf{    1.82} &    70.57 &    30.20 &    29.90 &    \revision{}{46.66} &     \revision{}{2.63} &     \revision{}{2.51} &     3.49 &     3.32 &     3.29\\
     & \pegaseLarge   & \textbf{    0.38} &    81.19 &    41.34 &    41.40 &    \revision{}{30.84} &     \revision{}{1.90} &     \revision{}{1.92} &     1.29 &     1.29 &     1.29\\
     & \pegaseXLarge  & \textbf{    0.28} &    76.32 &    30.00 &    30.02 &    \revision{}{69.23} &   \revision{}{25.09} &    \revision{}{25.09} &     1.81 &     1.81 &     1.81\\
     & \gocXXL        & \textbf{    0.45} &   136.25 &    53.34 &    53.41 &    \revision{}{41.34} &   \revision{}{ 22.53} &    \revision{}{22.43} &     2.35 &     2.31 &     2.31\\
\midrule
SSL  & \ieeeSmall     & \textbf{    0.78} &    45.56 &     4.50 &     4.34 &     \revision{}{2.82} &     \revision{}{2.79} &     \revision{}{2.79} &     2.80 &     2.78 &     2.78\\
     & \pegaseSmall   & \textbf{    0.68} &    64.69 &     4.56 &     4.44 &     \revision{}{9.83} &     \revision{}{3.97} &     \revision{}{3.95} &     2.61 &     1.87 &     1.87\\
     & \rteLarge      & \textbf{    1.68} &    55.16 &     9.76 &     9.43 &    \revision{}{48.57} &     \revision{}{8.79} &     \revision{}{8.53} &     3.04 &     2.75 &     2.70\\
     & \pegaseLarge   & \textbf{    0.29} &    44.74 &     4.33 &     4.33 &    \revision{}{42.06} &     \revision{}{2.28} &     \revision{}{2.30} &     0.93 &     0.66 &     0.66\\
     & \pegaseXLarge  & \textbf{    0.19} &    61.28 &    21.35 &    21.32 &    \revision{}{65.38} &    \revision{}{19.64} &    \revision{}{19.64} &     0.91 &     0.89 &     0.89\\
     & \gocXXL        & \textbf{    0.33} &    93.91 &    10.00 &     9.98 &     \revision{}{9.47} &     \revision{}{2.58} &     \revision{}{2.58} &     0.71 &     0.64 &     0.64\\
\bottomrule
\end{tabular}
}
    \end{table}

Tables \ref{tab:exp:ED:opt_gap:FFR} and \ref{tab:exp:ED-R:opt_gap:FFR}
further demonstrate the benefits of training end-to-end feasible
models: they report, for ED and ED-R problems, the optimality gaps
achieved by DNN, \revision{}{DeepOPF} and DC3 \emph{after applying a repair step at
  inference time}. Two repair mechanisms are compared: the proposed
Repair Layers (RL) and a Euclidean Projection (EP).  The tables also
report the mean gap achieved by E2ELR as a reference baseline.  The
results can be summarized as follows. First, the additional
feasibility restoration improves the quality of the initial
prediction.  This is especially true for DNN \revision{}{and DeepOPF}, which exhibited the
largest constraint violations (see Table \ref{tab:res:violation}):
optimality gaps are improved by a factor 2--20, but remain very high
nonetheless.  Second, the two repair mechanisms yield similar
optimality gaps.  For DC3, there is virtually no difference between RL
and EP.  Third, across all experiments, even after feasibility
restoration, E2ELR remains the best-performing model, with optimality
gaps 2--6x smaller than DC3. Table \ref{tab:exp:feasrec_time} compares
the computing times of the feasibility restoration using either the
repair layers (RL) or the euclidean projection (EP).  The latter is
solved as a quadratic program with Gurobi.  All benchmarks are
conducted in Julia on a single thread, using the
\texttt{BenchmarkTools} utility \cite{BenchmarkTools.jl-2016}, and
median times are reported.  The results of Table
\ref{tab:exp:feasrec_time} show that evaluating the proposed repair
layers is three orders of magnitude faster than solving the euclidean
projection problem.

    \begin{table}[!t]
        \centering
        \caption{Computing time of feasibility restoration using feasibility layers (FL) and Euclidean projection (EP).}
        \label{tab:exp:feasrec_time}
        \begin{tabular}{llrrr}
            \toprule
            Problem  & System           & \multicolumn{1}{c}{RL} & \multicolumn{1}{c}{EP} & Speedup \\
            \midrule
            ED    & \ieeeSmall       & \qty[mode=text]{  0.13}{ \us} & \qty[mode=text]{  0.45}{ \ms} & 3439x\\
                     & \pegaseSmall     & \qty[mode=text]{  0.55}{ \us} & \qty[mode=text]{  1.41}{ \ms} & 2572x\\
                     & \rteLarge        & \qty[mode=text]{  1.40}{ \us} & \qty[mode=text]{  3.75}{ \ms} & 2686x\\
                     & \pegaseLarge     & \qty[mode=text]{  2.37}{ \us} & \qty[mode=text]{  6.90}{ \ms} & 2911x\\
                     & \pegaseXLarge    & \qty[mode=text]{  6.42}{ \us} & \qty[mode=text]{ 20.71}{ \ms} & 3227x\\
                     & \gocXXL          & \qty[mode=text]{  5.67}{ \us} & \qty[mode=text]{ 17.87}{ \ms} & 3155x\\
            \midrule
            ED-R  & \ieeeSmall       & \qty[mode=text]{  1.06}{ \us} & \qty[mode=text]{  1.00}{ \ms} & 939x\\
                     & \pegaseSmall     & \qty[mode=text]{  4.58}{ \us} & \qty[mode=text]{  3.42}{ \ms} & 748x\\
                     & \rteLarge        & \qty[mode=text]{ 10.46}{ \us} & \qty[mode=text]{ 10.19}{ \ms} & 974x\\
                     & \pegaseLarge     & \qty[mode=text]{ 20.14}{ \us} & \qty[mode=text]{ 18.38}{ \ms} & 913x\\
                     & \pegaseXLarge    & \qty[mode=text]{ 42.67}{ \us} & \qty[mode=text]{ 60.73}{ \ms} & 1423x\\
                     & \gocXXL          & \qty[mode=text]{ 39.80}{ \us} & \qty[mode=text]{ 49.17}{ \ms} & 1236x\\
            \bottomrule
        \end{tabular}\\
        \footnotesize{Median computing times as measured by \texttt{BenchmarkTools}}
    \end{table}

\section{Conclusion}
\label{sec:conclusion}

The paper proposed a new \emph{End-to-End Learning and Repair} (E2ELR)
architecture for training optimization proxies for economic dispatch
problems. E2ELR combines deep learning with closed-form, differential
repair layers, thereby integrating prediction and feasibility
restoration in an end-to-end fashion. The E2ELR architecture can be
trained with self-supervised learning, removing the need for
labeled data and the solving of numerous optimization problems
offline. The paper conducted extensive numerical experiments on the
ecocomic dispatch of large-scale, industry-size power grids with tens
of thousands of buses.  It also presented the first study that
considers reserve requirements in the context of optimization proxies,
reducing the gap between academic and industry formulations.
The results demonstrate that the combination of E2ELR and
self-supervised learning achieves state-of-the-art performance, with
optimality gaps that outperform other baselines by at least an order
of magnitude.  Future research will investigate security-constrained
economic dispatch (SCED) formulations, and the extension of repair
layers to \revision{}{thermal constraints, multi-period settings and the nonlinear, non-convex AC-OPF.}

\section*{Acknowledgments}

This research is partly funded by NSF awards 2007095 and 2112533, and by ARPA-E PERFORM award AR0001136.

\bibliographystyle{IEEEtran}
\bibliography{refs.bib}

\begin{thebibliography}{10}
\providecommand{\url}[1]{#1}
\csname url@samestyle\endcsname
\providecommand{\newblock}{\relax}
\providecommand{\bibinfo}[2]{#2}
\providecommand{\BIBentrySTDinterwordspacing}{\spaceskip=0pt\relax}
\providecommand{\BIBentryALTinterwordstretchfactor}{4}
\providecommand{\BIBentryALTinterwordspacing}{\spaceskip=\fontdimen2\font plus
\BIBentryALTinterwordstretchfactor\fontdimen3\font minus
  \fontdimen4\font\relax}
\providecommand{\BIBforeignlanguage}[2]{{%
\expandafter\ifx\csname l@#1\endcsname\relax
\typeout{** WARNING: IEEEtran.bst: No hyphenation pattern has been}%
\typeout{** loaded for the language `#1'. Using the pattern for}%
\typeout{** the default language instead.}%
\else
\language=\csname l@#1\endcsname
\fi
#2}}
\providecommand{\BIBdecl}{\relax}
\BIBdecl

\bibitem{BPM_002}
MISO, ``Energy and operating reserve markets,'' 2022, {B}usiness Practices
  Manual Energy and Operating Reserve Markets.

\bibitem{MISO2023_ReliabilityImperative}
\BIBentryALTinterwordspacing
{Midcontinent ISO}, ``{MISO's response to the reliability imperative},'' 1
  2023. [Online]. Available:
  \url{https://cdn.misoenergy.org/MISO\%20Response\%20to\%20the\%20Reliability\%20Imperative504018.pdf}
\BIBentrySTDinterwordspacing

\bibitem{Stover2023_ReliabilityRiskMetrics}
O.~Stover, P.~Karve, and S.~Mahadevan, ``Reliability and risk metrics to assess
  operational adequacy and flexibility of power grids,'' \emph{Reliability
  Engineering \& System Safety}, vol. 231, p. 109018, 2023.

\bibitem{Stover2022_JITRALF}
O.~Stover, P.~Karve, S.~Mahadevan, W.~Chen, H.~Zhao, M.~Tanneau, and P.~V.
  Hentenryck, ``{Just-In-Time Learning for Operational Risk Assessment in Power
  Grids},'' 2022.

\bibitem{Chen2022_Learning4SCED}
W.~Chen, S.~Park, M.~Tanneau, and P.~Van~Hentenryck, ``{Learning Optimization
  Proxies for Large-Scale Security-Constrained Economic Dispatch},''
  \emph{Electric Power Systems Research}, vol. 213, p. 108566, 2022.

\bibitem{Ng2018_StatisticalLearningDCOPF}
Y.~Ng, S.~Misra, L.~A. Roald, and S.~Backhaus, ``Statistical learning for dc
  optimal power flow,'' in \emph{2018 Power Systems Computation Conference
  (PSCC)}, 2018, pp. 1--7.

\bibitem{pan2020deepopf}
X.~Pan, T.~Zhao, M.~Chen, and S.~Zhang, ``{DeepOPF: A Deep Neural Network
  Approach for Security-Constrained DC OPF},'' \emph{IEEE Transactions on Power
  Systems}, vol.~36, no.~3, pp. 1725--1735, 2020.

\bibitem{Nellikkath2021_PINN-DCOPF}
R.~Nellikkath and S.~Chatzivasileiadis, ``Physics-informed neural networks for
  minimising worst-case violations in dc optimal power flow,'' in \emph{IEEE
  International Conference on Communications, Control, and Computing
  Technologies for Smart Grids (SmartGridComm)}, 2021, pp. 419--424.

\bibitem{zhao2022ensuring}
T.~Zhao, X.~Pan, M.~Chen, and S.~Low, ``{Ensuring DNN Solution Feasibility for
  Optimization Problems with Linear Constraints},'' in \emph{The Eleventh
  International Conference on Learning Representations}, 2023.

\bibitem{stratigakos2023interpretable}
\BIBentryALTinterwordspacing
A.~Stratigakos, S.~Pineda, J.~M. Morales, and G.~Kariniotakis, ``{Interpretable
  Machine Learning for DC Optimal Power Flow with Feasibility Guarantees},''
  Mar. 2023, working paper or preprint. [Online]. Available:
  \url{https://hal.science/hal-04038380}
\BIBentrySTDinterwordspacing

\bibitem{Ferrando2023_PERFORM_NYU}
\BIBentryALTinterwordspacing
R.~Ferrando, L.~Pagnier, R.~Mieth, Z.~Liang, Y.~Dvorkin, D.~Bienstock, and
  M.~Chertkov, ``{A Physics-Informed Machine Learning for Electricity Markets:
  A NYISO Case Study},'' 2023. [Online]. Available:
  \url{https://doi.org/10.48550/arXiv.2304.00062}
\BIBentrySTDinterwordspacing

\bibitem{Guha2019_ML4ACOPF}
N.~Guha, Z.~Wang, M.~Wytock, and A.~Majumdar, ``{Machine Learning for AC
  Optimal Power Flow},'' 2019.

\bibitem{Zamzam2020_LearningOPF}
A.~S. Zamzam and K.~Baker, ``{Learning Optimal Solutions for Extremely Fast AC
  Optimal Power Flow},'' in \emph{2020 IEEE International Conference on
  Communications, Control, and Computing Technologies for Smart Grids
  (SmartGridComm)}, 2020, pp. 1--6.

\bibitem{Owerko2020_OPFusingGNN}
D.~Owerko, F.~Gama, and A.~Ribeiro, ``{Optimal Power Flow Using Graph Neural
  Networks},'' in \emph{ICASSP 2020 - 2020 IEEE International Conference on
  Acoustics, Speech and Signal Processing (ICASSP)}, 2020, pp. 5930--5934.

\bibitem{fioretto2020predicting}
F.~Fioretto, T.~W. Mak, and P.~Van~Hentenryck, ``{Predicting AC Optimal Power
  Flows: Combining deep learning and lagrangian dual methods},'' in
  \emph{Proceedings of the AAAI conference on artificial intelligence},
  vol.~34, no.~01, 2020, pp. 630--637.

\bibitem{chatzos2021spatial}
M.~Chatzos, T.~W. Mak, and P.~Van~Hentenryck, ``{Spatial network decomposition
  for fast and scalable AC-OPF learning},'' \emph{IEEE Transactions on Power
  Systems}, vol.~37, no.~4, pp. 2601--2612, 2021.

\bibitem{donti2021dc3}
P.~Donti, D.~Rolnick, and J.~Z. Kolter, ``{DC3: A learning method for
  optimization with hard constraints},'' in \emph{International Conference on
  Learning Representations}, 2021.

\bibitem{nellikkath2022_PINN-ACOPF}
R.~Nellikkath and S.~Chatzivasileiadis, ``Physics-informed neural networks for
  ac optimal power flow,'' \emph{Electric Power Systems Research}, vol. 212, p.
  108412, 2022.

\bibitem{pan2022deepopf}
X.~Pan, M.~Chen, T.~Zhao, and S.~H. Low, ``{DeepOPF: A Feasibility-Optimized
  Deep Neural Network Approach for AC Optimal Power Flow Problems},''
  \emph{IEEE Systems Journal}, vol.~17, no.~1, pp. 673--683, 2023.

\bibitem{pan2022deepopf-AL}
X.~Pan, W.~Huang, M.~Chen, and S.~H. Low, ``{DeepOPF-AL: Augmented Learning for
  Solving AC-OPF Problems with a Multi-Valued Load-Solution Mapping},'' in
  \emph{Proceedings of the 14th ACM International Conference on Future Energy
  Systems}, ser. e-Energy '23.\hskip 1em plus 0.5em minus 0.4em\relax New York,
  NY, USA: Association for Computing Machinery, 2023, p. 42–47.

\bibitem{Zhou2022_DeepOPF-FT}
M.~Zhou, M.~Chen, and S.~H. Low, ``{DeepOPF-FT: One Deep Neural Network for
  Multiple AC-OPF Problems With Flexible Topology},'' \emph{IEEE Transactions
  on Power Systems}, vol.~38, no.~1, pp. 964--967, 2022.

\bibitem{liu2022topology}
S.~Liu, C.~Wu, and H.~Zhu, ``{Topology-aware Graph Neural Networks for Learning
  Feasible and Adaptive AC-OPF Solutions},'' \emph{IEEE Transactions on Power
  Systems}, pp. 1--11, 2022.

\bibitem{falconer2022leveraging}
T.~Falconer and L.~Mones, ``{Leveraging Power Grid Topology in Machine Learning
  Assisted Optimal Power Flow},'' \emph{IEEE Transactions on Power Systems},
  pp. 1--13, 2022.

\bibitem{Owerko2022_unsupervisedOPFusingGNN}
D.~Owerko, F.~Gama, and A.~Ribeiro, ``{Unsupervised Optimal Power Flow Using
  Graph Neural Networks},'' 2022.

\bibitem{Pham2022_ReduceOPFwithGNN}
T.~Pham and X.~Li, ``{Reduced Optimal Power Flow Using Graph Neural Network},''
  in \emph{2022 North American Power Symposium (NAPS)}, 2022, pp. 1--6.

\bibitem{gao2023physics}
M.~Gao, J.~Yu, Z.~Yang, and J.~Zhao, ``{A Physics-Guided Graph Convolution
  Neural Network for Optimal Power Flow},'' \emph{IEEE Transactions on Power
  Systems}, 2023.

\bibitem{park2023compact}
S.~Park, W.~Chen, T.~W. Mak, and P.~Van~Hentenryck, ``{Compact Optimization
  Learning for AC Optimal Power Flow},'' \emph{arXiv:2301.08840}, 2023.

\bibitem{mitrovic2023data}
M.~Mitrovic, A.~Lukashevich, P.~Vorobev, V.~Terzija, S.~Budennyy, Y.~Maximov,
  and D.~Deka, ``{Data-driven stochastic AC-OPF using Gaussian process
  regression},'' \emph{International Journal of Electrical Power \& Energy
  Systems}, vol. 152, p. 109249, 2023.

\bibitem{gupta2022dnn}
S.~Gupta, S.~Misra, D.~Deka, and V.~Kekatos, ``{DNN-based policies for
  stochastic AC OPF},'' \emph{Electric Power Systems Research}, vol. 213, p.
  108563, 2022.

\bibitem{Klamkin2022_ActiveBucketizedSampling}
M.~Klamkin, M.~Tanneau, T.~W.~K. Mak, and P.~Van~Hentenryck, ``{Active
  Bucketized Learning for ACOPF Optimization Proxies},'' 2022.

\bibitem{Hu2023_OPFWorthLearning}
Z.~Hu and H.~Zhang, ``Optimal power flow based on physical-model-integrated
  neural network with worth-learning data generation,'' 2023.

\bibitem{Huang2021_DeepOPF-NGT}
\BIBentryALTinterwordspacing
W.~Huang and M.~Chen, ``{DeepOPF-NGT: Fast No Ground Truth Deep Learning-Based
  Approach for AC-OPF Problems},'' in \emph{ICML 2021 Workshop Tackling Climate
  Change with Machine Learning}, 2021. [Online]. Available:
  \url{https://www.climatechange.ai/papers/icml2021/18}
\BIBentrySTDinterwordspacing

\bibitem{wang2022fast}
J.~Wang and P.~Srikantha, ``Fast optimal power flow with guarantees via an
  unsupervised generative model,'' \emph{IEEE Transactions on Power Systems},
  2022.

\bibitem{park2023self}
S.~Park and P.~Van~Hentenryck, ``Self-supervised primal-dual learning for
  constrained optimization,'' in \emph{Proceedings of the AAAI Conference on
  Artificial Intelligence}, vol.~37, no.~4, 2023, pp. 4052--4060.

\bibitem{Venzke2021_DNNVerificationPowerSystems}
A.~Venzke and S.~Chatzivasileiadis, ``Verification of neural network behaviour:
  Formal guarantees for power system applications,'' \emph{IEEE Transactions on
  Smart Grid}, vol.~12, no.~1, pp. 383--397, 2021.

\bibitem{Taheri2023_RestoringACOPFFeasibility}
B.~Taheri and D.~K. Molzahn, ``Restoring ac power flow feasibility from relaxed
  and approximated optimal power flow models,'' 2023.

\bibitem{Agrawal2019_CVXLayer}
A.~Agrawal, B.~Amos, S.~Barratt, S.~Boyd, S.~Diamond, and J.~Z. Kolter,
  ``Differentiable convex optimization layers,'' \emph{Advances in neural
  information processing systems}, vol.~32, 2019.

\bibitem{kim2022projection}
M.~Kim and H.~Kim, ``Projection-aware deep neural network for dc optimal power
  flow without constraint violations,'' in \emph{2022 IEEE International
  Conference on Communications, Control, and Computing Technologies for Smart
  Grids (SmartGridComm)}, 2022, pp. 116--121.

\bibitem{Li2022_LOOP-LC}
M.~Li, S.~Kolouri, and J.~Mohammadi, ``{Learning to Solve Optimization Problems
  With Hard Linear Constraints},'' \emph{IEEE Access}, vol.~11, pp.
  59\,995--60\,004, 2023.

\bibitem{Josz2016_ACOPF_PegaseRTE}
C.~Josz, S.~Fliscounakis, J.~Maeght, and P.~Panciatici, ``{AC Power Flow Data
  in MATPOWER and QCQP Format: iTesla, RTE Snapshots, and PEGASE},'' 2016.

\bibitem{Holzer2022_MISO_SFT}
\BIBentryALTinterwordspacing
J.~Holzer, Y.~Chen, Z.~Wu, F.~Pan, and A.~Veeramany, ``{Fast Simultaneous
  Feasibility Test for Security Constrained Unit Commitment},'' 2022. [Online].
  Available: \url{https://dx.doi.org/10.36227/techrxiv.20280384.v1}
\BIBentrySTDinterwordspacing

\bibitem{Ma2009_MISO_SCED}
X.~Ma, H.~Song, M.~Hong, J.~Wan, Y.~Chen, and E.~Zak, ``The
  security-constrained commitment and dispatch for midwest iso day-ahead
  co-optimized energy and ancillary service market,'' in \emph{2009 IEEE Power
  \& Energy Society General Meeting}, 2009, pp. 1--8.

\bibitem{BPM002_D}
MISO, ``Real-time energy and operating reserve market software formulations and
  business logic,'' 2022, {B}usiness Practices Manual Energy and Operating
  Reserve Markets Attachment D.

\bibitem{Chen2023_E2ELR_arxiv}
W.~Chen, M.~Tanneau, and P.~Van~Hentenryck, ``{End-to-End Feasible Optimization
  Proxies for Large-Scale Economic Dispatch},'' \emph{arXiv (preprint)}, 2023.

\bibitem{PGLib}
S.~Babaeinejadsarookolaee, A.~Birchfield, R.~D. Christie, C.~Coffrin,
  C.~DeMarco, R.~Diao, M.~Ferris, S.~Fliscounakis, S.~Greene, R.~Huang,
  C.~Josz, R.~Korab, B.~Lesieutre, J.~Maeght, T.~W.~K. Mak, D.~K. Molzahn,
  T.~J. Overbye, P.~Panciatici, B.~Park, J.~Snodgrass, A.~Tbaileh,
  P.~Van~Hentenryck, and R.~Zimmerman, ``{The Power Grid Library for
  Benchmarking AC Optimal Power Flow Algorithms},'' 2019.

\bibitem{UW_PowerSystemArchive}
\BIBentryALTinterwordspacing
{University of Washington, Dept. of Electrical Engineering}, ``{Power systems
  test case archive},'' 1999. [Online]. Available:
  \url{http://www.ee.washington.edu/research/pstca/}
\BIBentrySTDinterwordspacing

\bibitem{GoCompetition}
\BIBentryALTinterwordspacing
{Grid Optimization Competition}, ``{“Grid optimization competition
  datasets},'' 2018. [Online]. Available:
  \url{https://gocompetition.energy.gov/}
\BIBentrySTDinterwordspacing

\bibitem{MISO_Schedule28_ReserveDemandCurve}
\BIBentryALTinterwordspacing
{MISO}, ``{Schedule 28 -- Demand Curves for TOperating Reserves},'' 2023.
  [Online]. Available: \url{https://www.misoenergy.org/legal/tariff/}
\BIBentrySTDinterwordspacing

\bibitem{MISO_Schedule28A_TranmissionDemandCurve}
\BIBentryALTinterwordspacing
------, ``{Schedule 28A -- Demand Curves for Transmission Constraints},'' 2019.
  [Online]. Available: \url{https://www.misoenergy.org/legal/tariff/}
\BIBentrySTDinterwordspacing

\bibitem{Dunning2017_JuMP}
I.~Dunning, J.~Huchette, and M.~Lubin, ``Jump: A modeling language for
  mathematical optimization,'' \emph{SIAM Review}, vol.~59, no.~2, pp.
  295--320, 2017.

\bibitem{gurobi}
\BIBentryALTinterwordspacing
{Gurobi Optimization, LLC}, ``{Gurobi Optimizer Reference Manual},'' 2023.
  [Online]. Available: \url{https://www.gurobi.com}
\BIBentrySTDinterwordspacing

\bibitem{paszke2017automatic}
A.~Paszke, S.~Gross, S.~Chintala, G.~Chanan, E.~Yang, Z.~DeVito, Z.~Lin,
  A.~Desmaison, L.~Antiga, and A.~Lerer, ``Automatic differentiation in
  pytorch,'' in \emph{NIPS-W}, 2017.

\bibitem{kingma2014adam}
D.~P. Kingma and J.~Ba, ``Adam: {A} method for stochastic optimization,'' in
  \emph{3rd International Conference on Learning Representations, {ICLR} 2015,
  San Diego, CA, USA, May 7-9, 2015, Conference Track Proceedings}, Y.~Bengio
  and Y.~LeCun, Eds., 2015.

\bibitem{PACE}
\BIBentryALTinterwordspacing
PACE, \emph{{P}artnership for an {A}dvanced {C}omputing {E}nvironment
  ({PACE})}, 2017. [Online]. Available: \url{http://www.pace.gatech.edu}
\BIBentrySTDinterwordspacing

\bibitem{BenchmarkTools.jl-2016}
J.~{Chen} and J.~{Revels}, ``{Robust benchmarking in noisy environments},''
  \emph{arXiv e-prints}, Aug 2016.

\bibitem{ioffe2015batch}
S.~Ioffe and C.~Szegedy, ``Batch normalization: Accelerating deep network
  training by reducing internal covariate shift,'' in \emph{International
  conference on machine learning}.\hskip 1em plus 0.5em minus 0.4em\relax pmlr,
  2015, pp. 448--456.

\bibitem{srivastava2014dropout}
N.~Srivastava, G.~Hinton, A.~Krizhevsky, I.~Sutskever, and R.~Salakhutdinov,
  ``Dropout: a simple way to prevent neural networks from overfitting,''
  \emph{The journal of machine learning research}, vol.~15, no.~1, pp.
  1929--1958, 2014.

\end{thebibliography}

\clearpage
\appendix
\subsection{Proofs}
\label{app:proofs}

\subsubsection{Proof of Theorem \ref{thm:power_balance_layer}}

    Let $\mathbf{\tilde{p}} \, {=} \, \mathcal{P}(\pg)$, and assume $\mathbf{e}^{\top} \pg \, {<} \, D$.
    It is immediate that $\etaup \, {\in} \, [0, 1]$, i.e., $\mathbf{\tilde{p}}$ is a convex
    combination of $\pg$ and $\pgmax$.
    Thus, $\mathbf{\tilde{p}} \in
    \hypercube$.  Then,
        \begin{align*}
            \mathbf{e}^{\top}\mathbf{\tilde{p}}
            &= (1 - \etaup) \mathbf{e}^{\top}\pg + \etaup \mathbf{e}^{\top}\bar{\pg}\\
            &= \etaup(\mathbf{e}^{\top}\bar{\pg} - \mathbf{e}^{\top}\pg) + \mathbf{e}^{\top} \pg\\
            &= \frac{D - \mathbf{e}^{\top} \pg}{\mathbf{e}^{\top} \mathbf{\bar{p}} - \mathbf{e}^{\top}\pg} (\mathbf{e}^{\top}\bar{\pg} - \mathbf{e}^{\top}\pg) + \mathbf{e}^{\top} \pg\\
            &= D - \mathbf{e}^{\top}\pg + \mathbf{e}^{\top} \pg = D.
        \end{align*}
        Thus, $\mathbf{\tilde{p}}$ satisfies the power balance and $\mathbf{\tilde{p}} \in \hypersimplex{D}$.
        Similarly, assume $\mathbf{e}^{\top}\pg \geq D$.
        Then
        \begin{align*}
            \mathbf{e}^{\top}\mathbf{\tilde{p}}
            &= (1 - \etadn) \mathbf{e}^{\top}\pg + \etadn \mathbf{e}^{\top}\mathbf{0}\\
            &= (1 - \frac{\mathbf{e}^{\top}\pg - D}{\mathbf{e}^{\top} \pg}) \mathbf{e}^{\top}\pg = D,
        \end{align*}
        which concludes the proof.
        \hfill \qed

\subsubsection{Proof of Theorem \ref{thm:reserve_layer:guaranteed_feasibility}}

    Let $\pg$ be the initial prediction, and let $\mathbf{\tilde{p}} = \mathcal{R}(\pg)$.
    First, the bound constraints $\mathbf{0} \, {\leq} \, \pg \, {\leq} \, \mathbf{\bar{p}}$ are satisfied because $\mathbf{\tilde{p}}$ is a convex combination of $\pg$ and $\pgmax - \resmax$, both of which satisfy bound constraints.
    
    Second, we show that $\mathbf{e}^{\top}\pg = D$.
    Indeed, we have
    \begin{align}
        \sum_{g} \tilde{p}_{g} &= \sum_{g \in G^{\uparrow}} \tilde{p}_{g} + \sum_{g \in G^{\downarrow}} \tilde{p}_{g}
    \end{align}
    This yields
    \begin{align*}
        \sum_{g \in G^{\uparrow}} \tilde{p}_{g} &= \sum_{g \in G^{\uparrow}} (1 - \alpha^{\uparrow}) p_{g} + \alpha^{\uparrow} (\bar{p}_{g} - \bar{r}_{g})\\
        &= \sum_{g \in G^{\uparrow}} p_{g} + \alpha^{\uparrow} \sum_{g \in G^{\uparrow}} (\bar{p}_{g} - \bar{r}_{g} - p_{g})\\
        &= \sum_{g \in \uparrow} p_{g} + \alpha^{\uparrow} \Delta^{\uparrow}\\
        &= \sum_{g \in G^{\uparrow}} p_{g} + \Delta \\
        \sum_{g \in G^{\downarrow}} \tilde{p}_{g}
        &= \sum_{g \in G^{\downarrow}} (1 - \alpha^{\downarrow}) p_{g} + \alpha^{\downarrow} (\bar{p}_{g} - \bar{r}_{g})\\
        &= \sum_{g \in G^{\downarrow}} p_{g} + \alpha^{\downarrow} \sum_{g \in G^{\downarrow}} (\bar{p}_{g} - \bar{r}_{g} - p_{g})\\
        &= \sum_{g \in G^{\downarrow}} p_{g} - \alpha^{\downarrow} \Delta^{\downarrow}\\
        &= \sum_{g \in G^{\downarrow}} p_{g} - \Delta
    \end{align*}
    Summing the two terms yields $\mathbf{e}^{\top}\pg + \Delta - \Delta = D$.
        
    Finally, we show that Algorithm \ref{alg:FFR:reserves} provides a feasible point iff DCOPF is feasible.
    If $\mathbf{\tilde{p}}$ is reserve feasible, then DCOPF is trivially feasible.
    
    Assume $\mathbf{\tilde{p}}$ is infeasible, i.e., $\sum_{g} \min(\bar{r}_{g}, \bar{p} - \tilde{p}_{g}) < R$.
    Note that if $\Delta = \Delta_{R}$ in Algorithm \ref{alg:FFR:reserves}, then $\mathbf{\tilde{p}}$ is feasible, so we must have $\Delta^{\uparrow} < \Delta_{R}$ or $\Delta^{\downarrow} < \Delta_{R}$.
    
    Assume the former holds, i.e., $\Delta^{\uparrow} < \Delta_{R}, \Delta_{\downarrow}$.
    In other words:
    \begin{align*}
        \sum_{g \in G^{\uparrow}} \bar{p}_{g} - \bar{r}_{g} - p_{g} < R - \sum_{g \in G^{\uparrow}} \bar{r}_{g} - \sum_{g \in G^{\downarrow}} \bar{p}_{g} - p_{g}\\
        \sum_{g \in G^{\uparrow}} (\bar{p}_{g} - \bar{r}_{g} - p_{g} + \bar{r}_{g}) < R - \sum_{g \in G^{\downarrow}} \bar{p}_{g} - p_{g}\\
        \sum_{g \in G^{\uparrow}} (\bar{p}_{g} - p_{g}) < R - \sum_{g \in G^{\downarrow}} \bar{p}_{g} - p_{g}\\
        \sum_{g \in G^{\uparrow}} \bar{p}_{g} + \sum_{g \in G^{\downarrow}} \bar{p}_{g} < R - \sum_{g} p_{g}\\
        \sum_{g} \bar{p}_{g} < R - D
    \end{align*}
    Summing constraints \eqref{eq:DCOPF:eco_max} and \eqref{eq:DCOPF:power_balance}, we obtain
    \begin{align*}
        \sum_{g} r_{g} \leq (\sum_{g} \bar{p}_{g}) - D < R
    \end{align*}
    hence DCOPF is infeasible.
    The proof for $\Delta^{\downarrow} < \Delta_{R}$ is similar.
    \hfill \qed
\subsection{Details of the baseline models}
\label{sec:appendix:baseline}

Both DC3 \cite{donti2021dc3} and LOOP-LC \cite{Li2022_LOOP-LC} follow the steps of neural network prediction, inequality correction, and equality completion.
First, the decision variables are divided into two groups: $|\mathcal{G}|-N_{eq}$ independent decision variables and $N_{eq}$ dependent decision variables, where $N_{eq}$ indicates the number of equality constraints. 
In the PTDF formulation of DC-OPF, the only equality constraint is the power balance constraint \eqref{eq:DCOPF:power_balance}, and thus $N_{eq}=1$. 
Therefore, given the dispatches of the independent generator are predicted, the dispatch of the dependent generator can be recovered by 
\begin{equation}
    p_1 = D - \sum_{g\in{\mathcal{G}}\setminus1}p_g. \label{eq:dc3:recover_den_variable}
\end{equation}
DC3 and LOOP-LC differ in their inequality corrections.

\paragraph{DC3}
Given the input load profile $\mathbf{l} \in \mathbb{R}^{|\mathcal{L}|}$, the neural network outputs $\mathbf{z} \in [0, 1]^{|\mathcal{G}|-1}$.
This is achieved by applying a sigmoid function to the final layer of the network.
Then the capacity constraints \eqref{eq:DCOPF:eco_max} are enforced by:
\[p_g = z_g*\bar{p}_g, \;\; \forall g \in \mathcal{G}\setminus 1.\]

In the inequality correction steps, DC3 minimizes the constraint violation by unrolling gradient descent with a fixed number of iterations $T$.
Denote the constraint violation \[g(\mathbf{p}) = \sum_{g\in \mathcal{G}} \max(p_g - \bar{p}_g, 0) + \max (R - \sum_{g\in \mathcal{G}} r_g, 0),\]
where $r_g = \min\{\bar{r}_{g}, \bar{p}_{g} - p_{g}\}$ and $p_1 = D - \sum_{g\in{\mathcal{G}}\setminus1}p_g$.
The dispatch is updated using: 
\begin{align*}
    \mathbf{p}^{t} = \mathbf{p}^{t-1} - \rho*\nabla_\mathbf{p}\|g(\mathbf{p}^{t-1})\|^2_2,
\end{align*}
where $\mathbf{p}^{0}$ is the output of the neural network. 
In the experiment, the step size $\rho$ is set as $1e-4$ and the total iteration $T$ is set as $50$ when training and $200$ when testing.
The longer testing $T$ is suggested in \cite{donti2021dc3} to mitigate the constraint violation of DC3 predictions.

\paragraph{LOOP-LC}
Similar to DC3, the neural network maps the load profile $\mathbf{l} \in \mathbb{R}^{|\mathcal{L}|}$ to $\mathbf{z} \in [0, 1]^{|\mathcal{G}|-1}$ by applying a sigmoid function at the end.
In the inequality correction step, LOOP-LC uses gauge function mapping $\mathbf{z}$ in the $l_\infty$ norm ball to the dispatches in the feasible region.
The gauge mapping needs an interior point to shift the domain.
The work in \cite{Li2022_LOOP-LC} proposes an interior point finder by solving an optimization, which could be computationally expensive.
Instead, the experiments exploit the proposed feasibility restoration layers to find the interior point effectively. 
Specifically, the interior point finder consists of two steps.
First, the optimal dispatches of the nominal case $\mathbf{p}^{n}$ are obtained by solving the instance with the nominal active power demand as the input, where the upper bounds $\bar{p}_{g}$ in constraints \eqref{eq:DCOPF:eco_max} and \eqref{eq:DCOPF:dispatch_bounds} are scaled with $\beta \in (0,1)$: $\pg + \res \leq \beta \pgmax$ and $\mathbf{0} \leq \pg \leq \beta \pgmax$,
where the $\beta$ is set as $0.8$ in the experiments.
The scaling aims at providing a more interior point such that the gauge mapping is more smooth.
However, the $\mathbf{p}^{n}$ may not be feasible when changing the input load profile $\mathbf{l}$.
To obtain the feasible solution, the proposed feasibility layers are used to convert $\mathbf{p}^{n}$ to an interior point.

\subsection{Hyperparameter Tuning}
\label{sec:appendix:hyperparameters}

For each test case and method, the number of instances in the training and test minibatch are set to 64 and 256, respectively.
For all deep learning models, a batch normalization layer \cite{ioffe2015batch} and a dropout layer \cite{srivastava2014dropout} with a dropout rate $0.2$ are appended after each dense layer except the last one.
The number of layers $l$ is selected from $\{3, 4, 5\}$ and the hidden dimension $hd$ of the dense layers is selected from $\{128, 256\}$.

The penalty coefficients $\lambda$ of the constraint violation in the loss function \ref{eq:SL:loss:final} and \ref{eq:SSL:loss} are selected from $\{1, 0.1\}$ for self-supervised learning and selected from $\{1e-3, 1e-4, 1e-5\}$ for supervised learning. For the models with feasibility guarantees such as DNN-F and LOOP, the $\lambda$ is set as 0 for self-supervised learning. $\mu$ is set to be equal to $\lambda$ for supervised learning.
For DC3 model, the unrolled iteration is set as $50$ iterations in training and $200$ iterations in testing. The gradient step size is set as $1e-4$, where a larger step size results in numerical issues.

The models are trained with Adam optimizer \cite{kingma2014adam} with an initial learning rate is set as $1e-2$ and weight delay $1e-6$. The learning rate is decayed by 0.1 when the validation loss does not improve for consecutive 10 epochs and the training early stops if the validation loss does not decrease for consecutive 20 epochs. The maximum training time is set as 150 minutes.



\end{document}